\documentclass[11pt,a4paper]{article}

\usepackage[a4paper,top=5cm,bottom=5cm,left=4cm,right=4cm]{geometry}

\usepackage[latin1]{inputenc}
\usepackage[T1]{fontenc}
\usepackage[english]{babel}

\usepackage{sansmath}
\usepackage{libertine}
\usepackage{libertinust1math}

\usepackage{calrsfs}
\usepackage{enumitem}

\setlist{nolistsep}
\usepackage{mathtools}
\usepackage{mathrsfs}

\usepackage{mfirstuc}
\usepackage{multicol}

\usepackage[overload]{textcase}
\usepackage{tasks}
\settasks{counter-format=tsk[1].}
\usepackage{bbm}
\usepackage{amsmath}

\usepackage{amssymb}
\usepackage{amsfonts}
\usepackage{amsthm}
\usepackage{wasysym}

\usepackage{graphicx}
\usepackage{wrapfig}
\usepackage[outdir=./]{epstopdf}
\usepackage[labelfont={bf,it},textfont=it]{caption}
\usepackage[hidelinks]{hyperref}

\usepackage{widetable}
\usepackage[scr=pxtx, scrscaled=1]{mathalfa}
\usepackage{xfrac}
\usepackage{soul}
\usepackage{tikz}
\usepackage{setspace}
\usepackage{yfonts}

\DeclareMathOperator{\Area}{Area}

\DeclareMathOperator{\Per}{Per}

\newcommand{\grad}{\nabla}
\newcommand{\M}{\mathcal{M}}
\newcommand{\IP}{\mathbb{P}}
\newcommand{\IE}{\mathbb{E}}
\newcommand{\IR}{\mathbb{R}}
\newcommand{\IZ}{\mathbb{Z}}
\newcommand{\IN}{\mathbb{N}}
\newcommand{\marknorm}{\ell}
\newcommand{\norm}[1]{\lVert #1\rVert}
\newcommand{\abs}[1]{\lvert #1\rvert}

\newcommand{\tg}{\textswab}

\newcommand{\msb}{\mathsfbf}
\newcommand{\temp}{\textit{temp}}

\newcommand*{\1}{\textrm{{\usefont{U}{fplmbb}{m}{n}1}}}

\newcommand{\defeq}{%
	\mathrel{ \vcenter{\baselineskip0.5ex \lineskiplimit0pt \hbox{\scriptsize.} \hbox{\scriptsize.}} } =}
\renewcommand{\eqdef}{%
	\mathrel{= \vcenter{\baselineskip0.5ex \lineskiplimit0pt \hbox{\scriptsize.} \hbox{\scriptsize.}} } }

\DeclareSymbolFont{CMletters}{OML}{cmm}{m}{it}
\DeclareMathSymbol{\xi}{\mathord}{CMletters}{"18}
\renewcommand{\Xi}{\varXi}
\renewcommand{\epsilon}{\varepsilon}
\renewcommand{\theta}{\vartheta}
\renewcommand{\Theta}{\varTheta}
\renewcommand{\Lambda}{\varLambda}
\renewcommand{\Delta}{\varDelta}

\theoremstyle{plain}
\theoremstyle{definition}
\newtheorem{Definition}{Definition}[section]
\newtheorem{Example}{Example}

\newtheorem{Lemma}[Definition]{Lemma}

\newtheorem{Proposition}[Definition]{Proposition}
\theoremstyle{plain}
\newtheorem{Theorem}{Theorem}%[Definition]
\theoremstyle{remark}
\newtheorem*{Remark}{Remark}
\newtheorem*{Remarks}{Remarks}

\usepackage[sc,explicit,pagestyles,raggedright]{titlesec}
\usepackage{textcase}
\usepackage{microtype}

\usepackage{fancyhdr}
\usepackage{titling}
\usepackage{lettrine}
\usepackage{sectsty}%needed for \allsectionsfont
\usepackage{xcolor}

\allsectionsfont{\large\itshape}
\subsectionfont{\normalfont\itshape\bfseries}

\usepackage{imakeidx}
\makeindex[columns=2,options=-s Template]

\newcommand*\myrule[1][.25\textwidth]{%
	\tikz {\path [fill, draw] (0,0) [out=0, in=180] to +(.5*#1,1pt) [out=0, in=180] to +(.5*#1,-1pt) [out=180, in=0] to +(-.5*#1,-1pt) [out=180, in=0] to cycle;}}

\makeatletter
\def\@maketitle{%
  \newpage
  \null
  \vskip 2em%
  \begin{center}%
  \let \footnote \thanks
    {\Large\textbf{\textit{ \@title }}\par}%
    \vskip 1.5em%
    {\large
      \lineskip .5em%
      \vskip -.8em
      \begin{tabular}[t]{c}%
        \@author%\@author
      \end{tabular}\par}%
    \vskip .5em%
    \par
  \vskip .7em
  \myrule
  \vskip 1.5em
  \end{center}%
  }
\makeatother

\renewenvironment{abstract}{%
\vskip -1.5em
\hfill\begin{center}\small
\begin{minipage}{0.9\textwidth}}
{\par\noindent
\end{minipage}
\end{center}
\vskip 2em}
\makeatother

\newenvironment{acknowledgments}{
\vskip 2em
\noindent\small\textbf{Acknowledgments:}}
{\par\noindent
\vskip 1.5em}
\makeatletter

\title{Marked Gibbs point processes with unbounded interaction: an existence result}
\author{Sylvie R\oe lly, Alexander Zass}
\begin{document}

	\maketitle
	\begin{abstract}
We construct marked Gibbs point processes in $\IR^d$ under quite general assumptions. Firstly, we allow for interaction functionals that may be unbounded and whose range is not assumed to be uniformly bounded. Indeed, our typical interaction admits an a.s. \emph{finite but random} range. Secondly, the random marks -- attached to the locations in $\IR^d$ -- belong to a general normed space $\mathcal{S}$. They are not bounded, but their law should admit a super-exponential moment. The approach used here relies on the so-called \emph{entropy method} and large-deviation tools in order to prove tightness of a family of finite-volume Gibbs point processes. An application to infinite-dimensional interacting diffusions is also presented.
\bigbreak
\smallbreak
\end{abstract}

\section{Introduction}

In this paper we construct a certain class of continuous marked Gibbs point processes. Recall that a marked point consists of a pair: a location $x\in\IR^d$, $d\geq 1$, and a mark $m$ belonging to a general normed space $\mathcal{S}$. The interactions we consider here are described by an energy functional $H$, which acts both on locations and on marks. This includes, in particular, the case of $k$-body potentials, but is indeed a more general framework, useful to treat examples coming from the field of stochastic geometry (as e.g. the area- or the Quermass-interaction model, see Example \ref{ex:quermass}).

The novelty of the results presented in this paper is threefold. Firstly, we do not assume a specific form of the interaction -- like pairwise or $k$-body -- but only make assumptions (in Section \ref{section:energy}) on the resulting energy functional $H$ itself. In particular, we do not assume superstability of the interaction, but only rely on the stability assumptions $(\mathcal{H}_{st})$ and $(\mathcal{H}_{loc.st})$. In the field of stochastic geometry, in particular, many quite natural energy functionals are stable but not superstable, like the Quermass-interaction model presented in Example \ref{ex:quermass}.

Secondly, the Gibbsian energy functional we consider has an \emph{unbounded range}: it is finite, but random and not uniformly bounded -- as opposed to models treated for example in \cite{conache_daletskii_kondratiev_pasurek_2018} which deal with a bounded-range interaction; see Assumption $(\mathcal{H}_{r})$. For a very recent existence proof in the case of infinite-range interaction (without marks) see \cite{dereudre_vasseur_2019}. Moreover, unlike the hyper-edge interactions presented in \cite{dereudre_drouilhet_georgii_2011}, we treat the case of interactions which are highly non local: the range of the conditional energy on a bounded region of an infinite configuration (see Definition \ref{def:locEnergy}) requires knowledge of the whole configuration and cannot be determined only by a local restriction of the configuration.

Lastly, we work with a mark reference distribution whose support is a priori \emph{unbounded} but only fulfils a super-exponential integrability condition (see Assumption $(\mathcal{H}_m)$).\\

Let us mention recent works on the existence of marked Gibbs point processes for \emph{particular} models. In \cite{dereudre_2009} D. Dereudre proves the existence of the \emph{Quermass-interaction process} as a planar germ-grain model; we draw inspiration from his approach, presenting here an existence result for more general processes, under weaker assumptions. In \cite{daletskii_kondratiev_kozitsky_pasurek_2014} and \cite{conache_daletskii_kondratiev_pasurek_2018} the authors treat the case of unbounded marks in $\IR^d$ with \emph{finite-range} energy functional which is induced by a \emph{pairwise interaction}.

\bigbreak

The main thread of our approach is the reduction of the general marked point process to a germ-grain model, where two marked points $(x_1,m_1),(x_2,m_2) \in \IR^d\times\mathcal{S}$ do not interact as soon as the balls with centre $x_i$ and radius $\norm{m_i},\ i=1,2,$ do not intersect. The framework we work in requires the introduction of a notion of \emph{tempered configurations} (see Section \ref{section:tempered}) in order to better control the support of the Gibbs measure we construct. In this way, the size growth of the marks of far away points is bounded. In Section \ref{section:support} we see that this procedure is justified by the fact that the constructed infinite-volume Gibbs measure is actually concentrated on tempered configurations.

The originality of our method to construct an \emph{infinite-volume measure} consists in the use of the specific entropy as a tightness tool. This relies on the fact that the level sets of the specific-entropy functional are relatively compact in the local convergence topology; see Section \ref{section:entropy}. This powerful topological property was first shown in the setting of marked point processes by H.-O. Georgii and H. Zessin in \cite{georgii_zessin_1993}. Indeed, we prove in Proposition \ref{prop:entropy}, using large-deviation tools, that the entropy of some sequence of finite-volume Gibbs measures is uniformly bounded. This sequence is therefore tight, and admits an accumulation point.
Let us remark that the entropy tool relies mainly on stability assumptions of the energy $H$, without the need for superstability. The usual approach (see e.g. \cite{ruelle_1970}), in fact, uses the superstability condition to precisely control the local density of points; in our framework, we do this thanks to an equi-integrability property, which holds on the entropy level sets (see Lemma \ref{lemma:EI}). Furthermore, the stability notion we use here is weaker than the classic one of Ruelle, as it includes a term depending on the marks of the configuration. For more details and examples, see Section \ref{section:energy}.
 
The last step of the proof consists in showing that this accumulation point satisfies the \emph{Gibbsian} property. Since the interaction is not local and not bounded, this property is not inherited automatically from the finite-volume approximations, but instead requires an accurate analysis, which is done in Section \ref{section:gibbs}. In Section \ref{section:diffusion} we propose an application to infinite-dimensional interacting diffusions.

\subsection{Point-measure formalism}

The point configurations considered here live in the product state space \mbox{$\mathcal{E}\defeq\mathbb{R}^d\times\mathcal{S}$}, $d\geq 1$, where $\big(\mathcal{S},\lVert\cdot\rVert\big)$ is a general normed space: each point \emph{location} in $\mathbb{R}^d$ has an associated \emph{mark} belonging to $\mathcal{S}$. The location space $\mathbb{R}^d$ is endowed with the Euclidean norm $\lvert\cdot\rvert$, and the associated Borel $\sigma$-algebra $\mathcal{B}(\IR^d$); we denote by $\mathcal{B}_b(\IR^d)\subset\mathcal{B}(\IR^d)$ the set of bounded Borel subsets of $\IR^d$. A set $\Lambda$ belonging to $\mathcal{B}_b(\IR^d)$ will often be called a \emph{finite volume}. We denote by $\mathcal{B}(\mathcal{S})$ the Borel $\sigma$-algebra on $\mathcal{S}$.

The set of point measures on $\mathcal{E}$ is denoted by $\M$; it consists of the integer-valued, $\sigma$-finite measures $\gamma$ on $\mathcal{E}$:
\begin{equation*}
  \M \defeq \big\{\gamma = \sum_i \delta_{\msb{x}_i}: \msb{x}_i=(x_i,m_i)\in\IR^d\times\mathcal{S}\big\}.
\end{equation*}
We endow $\mathcal{M}$ with the canonical $\sigma$-algebra generated by the family of local counting functions on $\mathcal{M}$,
\begin{equation*}
	\gamma=\sum_i\delta_{(x_i,m_i)}\mapsto\text{Card}(\{i: x_i\in\Lambda,\ m_i\in A\}),\ \Lambda\in\mathcal{B}_b(\IR^d),\ A\in\mathcal{B}(\mathcal{S}).
\end{equation*}

We denote by $\underline{o}$ the zero point measure whose support is the empty set. Since, in the framework developed in this paper, we only consider \emph{simple} point measures, we identify them with the subset of their atoms:
\begin{equation*}
  \gamma \equiv  \big\{\msb{x}_1,\dots,\msb{x}_n,\dots\big\} \subset\mathcal{E}.
\end{equation*}
For a point configuration  $\gamma \in \M$ and a fixed set $\Lambda\subset\IR^d$, we denote by $\gamma_\Lambda$ the restriction of the point measure $\gamma$ to the set $\Lambda\times\mathcal{S}$:
\begin{equation*}
  \gamma_\Lambda \defeq \gamma\cap(\Lambda\times\mathcal{S}) = \sum_{\{i:\, x_i\in\Lambda\}} \delta_{(x_i,m_i)}.
\end{equation*}
A \emph{functional} is a measurable $\IR\cup\{+\infty\}$-valued map defined on $\M$. We introduce specific notations for some of them: the mass of a point measure $\gamma$ is
denoted by $\lvert\gamma\rvert$. It corresponds to the number of its atoms if $\gamma$ is simple.\\
We also denote by $\tg{m}$ the supremum of the size of the marks of a configuration:
\begin{equation*}
  \tg{m}(\gamma) \defeq \sup\limits_{(x,m)\in\gamma} \norm{m}, \quad \gamma\in\M.
  \end{equation*}
  The integral of a fixed function $f\colon\mathcal{E}\rightarrow\IR$ under the measure $\gamma \in\M $ -- when it exists -- is denoted by
\begin{equation*}
  \langle\gamma,f\rangle \defeq \int f d\gamma = \sum_{\msb{x}\in\gamma}f(\msb{x}).
\end{equation*}
  For a finite volume $\Delta$, we call \emph{local} or more precisely $\Delta$-\emph{local},  any functional $F$ satisfying
  \begin{equation*}
    F(\gamma) = F(\gamma_\Delta), \quad \gamma\in\M.
  \end{equation*}
  We also define the set of finite point measures on $\mathcal{E}$:
\begin{equation*}
  \M_f \defeq \big\{ \gamma \in \M : \lvert\gamma\rvert < +\infty \big\}.
\end{equation*}
Moreover, for any bounded subset $\Lambda\subset\IR^d$, $\M_\Lambda$ is the subset of $\M_f$ consisting of the point measures whose support is included in $\Lambda\times\mathcal{S}$:
\begin{equation*}
  \M_\Lambda \defeq \big\{\gamma \in \M : \gamma = \gamma_\Lambda \big\}\subset\M_f.
\end{equation*}
Let $\mathcal{P}(\M)$ denote the set of probability measures on $\M$.\\
We write $\IN^*$ for the set of non-zero natural numbers $\IN\setminus\{0\}$. The open ball in $\IR^d$ centred in $y\in\IR^d$ with radius $r\in\IR_+$ is denoted by $B(y,r)$.

\section{Gibbsian setting}

\subsection{Mark reference distribution}\label{section:marks}

The mark associated to any point of a configuration is random. We assume that the reference mark distribution $R$ on $\mathcal{S}$ is such that its image under the map $m\mapsto\norm{m}$ is a probability measure $\rho$ on $\IR_+$ that admits a super-exponential moment, in the following sense:
\begin{description}
  \item[$(\mathcal{H}_m)$] There exits $\delta>0$ such that
  \begin{equation}\label{H:mark}
    \int_{\IR_+} e^{\marknorm^{d+2\delta}}\rho(d\marknorm)<+\infty.
  \end{equation}
\end{description}
Throughout Sections 2 and 3 of the paper, the parameter $\delta$ is fixed.
\begin{Remark}
  The probability measure $\rho$ is the density of a positive random variable $X$ such that $X^{\frac{2}{d}+\epsilon}$ is subgaussian for some $\epsilon>0$ (see e.g. \cite{kahane_1960}, \cite{ledoux_talagrand_2012}).
\end{Remark}

\subsection{Tempered configurations}\label{section:tempered}

We introduce the concept of \emph{tempered configuration}. For such a configuration  $\gamma$, the number of its points in any finite volume $\Lambda$,  $\lvert\gamma_\Lambda\rvert$, should grow sublinearly w.r.t. the volume, while its marks should grow as a fraction of it. More precisely, we define the space $\M^{\temp}$  of tempered configurations as the following increasing union
\begin{equation*}
  \M^{\temp} \defeq \bigcup\limits_{\tg{t}\in\IN}\M^{\tg{t}},
\end{equation*}
where
\begin{equation}\label{eq:temp}
  \M^{\tg{t}} = \big\{\gamma\in\M:\ \forall l\in\IN^*, \ \langle\gamma_{B(0,l)},f\rangle\leq\tg{t} \,l^d	\ \text{ for }f(x,m) \defeq 1+\norm{m}^{d+\delta}\big\}.
\end{equation}
We now prove some properties satisfied by tempered configurations.
\begin{Lemma}\label{lemma:temp}
  The mark associated to a point in a tempered configuration is asymptotically negligible with respect to the norm of the said point: any tempered configuration $\gamma\in\M^{\temp}$ satisfies
  \begin{equation*}
    \lim\limits_{l\rightarrow+\infty} \dfrac{1}{l} \tg{m}(\gamma_{B(0,l)}) = 0.
  \end{equation*}
\end{Lemma}

\begin{proof}
  Let $\gamma\in\M^{\tg{t}}$, $\tg{t}\geq 1$. From \eqref{eq:temp}, recalling that $\tg{m}(\gamma)=\sup\limits_{(x,m)\in\gamma}\norm{m}$, we get that, for all $l\geq1$,
  \begin{equation*}
    \tg{m}(\gamma_{B(0,l)})\leq \big(\tg{t} l^d\big)^{\sfrac{1}{d+\delta}} = \dfrac{(\tg{t} l^d)^{\sfrac{1}{d+\delta}}}{l} l.
  \end{equation*}
  Define, for any $\eta\in(0,1)$,
  \begin{equation}\label{eq:l1}
    l_1(\tg{t},\eta) \defeq \Big(\dfrac{\tg{t}}{\eta^{d+\delta}}\Big)^{\sfrac{1}{\delta}}.
  \end{equation}
  Then, if $l \geq l_1(\tg{t},\eta)$,
  \begin{equation}\label{eq:boundm}
    \dfrac{\tg{m}(\gamma_{B(0,l)})}{l} \leq \dfrac{(\tg{t} l^d)^{\sfrac{1}{d+\delta}}}{l}\leq\eta\in(0,1),
  \end{equation}
  and the Lemma is proved.
\end{proof}

\begin{Lemma}\label{lemma:range}
  Let $\gamma\in\M^{\tg{t}}, \tg{t}\geq 1$, and define $\tg{l}(\tg{t}) \defeq \tfrac{1}{2}\ l_1(\tg{t},\tfrac{1}{2})$, where $l_1$ is defined by \eqref{eq:l1}.
  Then, for all $l\geq\tg{l}(\tg{t})$, the following implication holds:
  \begin{equation*}
    \msb{x}=(x,m)\in\gamma_{B(0,2l+1)^c} \implies B(x,\norm{m})\cap B(0,l)=\emptyset.
  \end{equation*}
\end{Lemma}

\begin{proof}
  Let $\gamma\in\M^{\tg{t}}$ and $(x,m)\in\gamma$ such that $\lvert x\rvert \geq 2l +1$.

  By definition of $l_1(\tg{t},\tfrac{1}{2})$, since $(x,m)\in\gamma_{B(0,\lceil x\rceil)}$,
  \begin{equation*}
  \begin{split}
    \lvert x\rvert - \norm{m} \overset{\eqref{eq:boundm}}{\geq} \lvert x\rvert - \tfrac{1}{2}\lceil\lvert x\rvert\rceil \geq \tfrac{1}{2}\lvert x\rvert - \tfrac{1}{2}\geq l.
  \end{split}
  \end{equation*}
\end{proof}

\begin{figure}[ht]
  \hspace{2em}
  \begin{minipage}[c]{0.45\textwidth}
    \includegraphics[scale=.45]{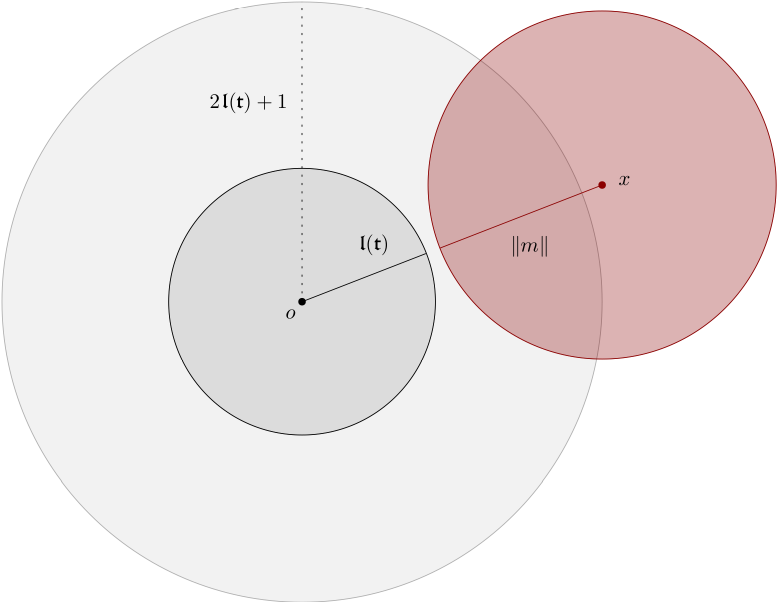}
  \end{minipage}\hfill
  \begin{minipage}[c]{0.38\textwidth}
    \caption{For \mbox{$(x,m)\in\gamma\in\M^\tg{t}$}, $\tg{t}\geq 1$, such that \mbox{$\lvert x\rvert\geq 2\, \tg{l}(\tg{t})+1$}, $B(x,\norm{m})$ does not intersect $B\big(0,\tg{l}(\tg{t})\big)$.}
    \label{fig:circles}
  \end{minipage}\hspace{2em}
\end{figure}

The assertion of Lemma \ref{lemma:range} is illustrated in Figure \ref{fig:circles}. Define the germ-grain set $\Gamma$ of a configuration $\gamma$ as usual by
  \begin{equation*}
    \Gamma \defeq \bigcup_{(x,m)\in\gamma} B(x,\norm{m})\subset\IR^d,
  \end{equation*}
  where the point $ x$ is the \emph{germ} and the ball $B(0,\norm{m})$ is the \emph{grain}. Lemma \ref{lemma:range} then implies that, for tempered configurations, only a finite number of balls of their germ-grain set can intersect a fixed bounded subset of $\IR^d$. This remark will be very useful when defining the range of the interaction in \eqref{H:range}.

\subsection{Energy functionals and finite-volume Gibbs measures}\label{section:energy}

For a fixed finite volume $\Lambda\subset\IR^d$, we consider, as reference marked point process, the Poisson point process $\pi_\Lambda^z$  on $\mathcal{E}$ with intensity measure \mbox{$z\ dx_\Lambda \otimes R(dm)$}. The coefficient $z$ is a positive real number, $dx_{\Lambda}$ is the Lebesgue measure on $\Lambda$, and the probability measure $R$ on $\mathcal{S}$ was introduced in Section \ref{section:marks}. In this toy model, since the spatial component of the intensity measure is diffuse, the configurations are a.s. simple. Moreover, the random marks of different points of the configuration are independent random variables.

To model and quantify a possible interaction between the point locations and the marks of a configuration, one introduces the general notion of energy functional.
\smallbreak

\begin{Definition}
  An \emph{energy} functional $H$ is a translation-invariant measurable functional on the space of finite configurations
  \begin{equation*}
    H:\M_f\rightarrow\IR\cup\{+\infty\}.
  \end{equation*}
   We use the convention $H(\underline{o}) = 0$.
\end{Definition}
\noindent Configurations with infinite energy will be negligible with respect to Gibbs measures.

\begin{Definition}
  For $\Lambda\in\mathcal{B}_b(\IR^d)$, the \emph{finite-volume Gibbs measure with free boundary condition} is the probability measure $P_\Lambda$ on $\M$ defined by
  \begin{equation}\label{eq:finiteGibbs}
    P_\Lambda(d\gamma) \defeq \dfrac{1}{Z_\Lambda}e^{-H(\gamma_{\Lambda})}\, \pi^z_{\Lambda}(d\gamma).
  \end{equation}
\end{Definition}
The normalisation constant $Z_\Lambda$ is called \emph{partition function}. We will see in Lemma \ref{lemma:partfunctionfinite} why this quantity is well defined under the assumptions we work with.

Notice how $\pi^z_{\Lambda}$ -- and therefore $P_\Lambda$ -- is actually concentrated on $\M_{\Lambda}$, the finite point configurations with atoms in $\Lambda$.

\vspace{2mm}
The measure $\pi_\Lambda^z$ extends naturally to an infinite-volume measure $\pi^z$; the question we explore in this work is how to do the same for $P_\Lambda$.
The first step in order to define an infinite-volume Gibbs measure is to be able to consider the energy of configurations with infinitely many points. In order to do this, we approximate any (tempered) configuration $\gamma$ by a sequence of finite ones $(\gamma_{\Lambda_n})_n$. Using a terminology that goes back to F\"ollmer \cite{foellmer_1973}, we introduce the following

\begin{Definition}\label{def:locEnergy}
  For $\Lambda\in\mathcal{B}_b(\IR^d)$, the \emph{conditional energy of $\gamma$ on $\Lambda$ given its environment} is the functional $H_\Lambda$ defined, on the tempered configurations, as the following limit:
  \begin{equation}\label{eq:locenergy}
    H_\Lambda(\gamma) = \lim\limits_{n\rightarrow\infty} \Big(H(\gamma_{\Lambda_n}) - H(\gamma_{\Lambda_n\setminus\Lambda})\Big), \quad \gamma\in\M^{\temp},
  \end{equation}
  where  $\Lambda_n \defeq [-n,n)^d$ is an increasing sequence of centred cubes of volume $(2n)^d$, converging to $\IR^d$.
\end{Definition}

\begin{Remarks}
  \begin{enumerate}[label=\emph{\roman*.}]
    \item Notice that the conditional energy of finite configurations confined in $\Lambda$ coincides with their energy: $H_\Lambda(\gamma_\Lambda)\equiv H(\gamma_\Lambda)$. In general, however, the conditional energy $H_\Lambda(\gamma)$ of an infinite configuration $\gamma$ does not reduce to $H(\gamma_\Lambda)$ because of the possible interaction between (external) points of $\gamma_{\Lambda^c}$ and (internal) points of $\gamma_\Lambda$. In other words, the conditional energy is possibly not a local functional. In this paper, we are interested in this general framework.
    \item Indeed, we will work with energy functionals $H$ for which the limit in  \eqref{eq:locenergy} is stationary, i.e. reached for a finite $n$ (that depends on $\gamma$). Assumption $(\mathcal{H}_{r})$ below ensures this property.
    \item Since $\pi^z_{\Lambda}$ only charges configurations in $\Lambda$, $P_\Lambda$ can be equivalently defined as
    \begin{equation*}
      P_\Lambda(d\gamma) = \dfrac{1}{Z_\Lambda}e^{-H_{\Lambda}(\gamma)}\pi^z_{\Lambda}(d\gamma).
    \end{equation*}
  \end{enumerate}
\end{Remarks}
The key property of such conditional energy functionals is the following additivity; the proof of this lemma is analogous to the one in \cite{dereudre_2009}, Lemma 2.4, that works in the more specific setting of Quermass-interaction processes.

\begin{Lemma}
  The family of conditional energy functionals is additive, i.e. for any $\Lambda\subset\Delta\in\mathcal{B}_b(\IR^d)$, there exists a measurable function $\phi_{\Lambda,\Delta}:\M^{\temp}\rightarrow\IR$ such that
  \begin{equation}\label{eq:additivity}
    H_{\Delta}(\gamma) =  H_{\Lambda}(\gamma) + \phi_{\Lambda,\Delta}(\gamma_{\Lambda^c}),
\quad
\gamma\in\M^{\temp}.
  \end{equation}
\end{Lemma}

Let us now describe the framework of our study, by considering for the energy functional $H$ a global stability assumption $(\mathcal{H}_{st})$, a range assumption $(\mathcal{H}_{r})$ and a locally-uniform stability assumption $(\mathcal{H}_{loc.st})$:
\begin{description}
  \item[$(\mathcal{H}_{st})$] There exists a constant $\tg{c}\geq0$ such that the following \emph{stability} inequality holds
  \begin{equation}\label{H:stab}
    H(\gamma)\geq -\tg{c} \, \langle\gamma,1+\norm{m}^{d+\delta}\rangle,
\quad \gamma\in\M_f.
  \end{equation}
  \item[$(\mathcal{H}_{r})$] Fix $\Lambda\in\mathcal{B}_b(\IR^d)$.
For any $\gamma\in\M^{\tg{t}},\ \tg{t}\geq 1$, there exists a positive finite number $\tg{r} = \tg{r}(\gamma,\Lambda) $ such that
  \begin{equation}\label{H:range}
    H_\Lambda(\gamma) = H\big(\gamma_{\Lambda\oplus B(0,\tg{r})}\big) - H\big(\gamma_{\Lambda\oplus B(0,\tg{r}))\setminus \Lambda}\big),
  \end{equation}
  where $\Lambda\oplus B(0,\tg{r}) \defeq \big\{x\in\IR^d : \exists y\in\Lambda,\ \lvert y-x\rvert \leq \tg{r} \big\}$, and 
  \begin{equation*}
    \tg{r}(\gamma,\Lambda) = 3 \tg{d}(\Lambda) + 2 \,\tg{l}(\tg{t}) + 2\, \tg{m}(\gamma_\Lambda) + 1,
  \end{equation*}
  where $\tg{d}(\Lambda)$ is the smallest $R>0$ such that $\Lambda\subset B(0,R)$. \footnote{An earlier version of this work stated the range assumption without the $3 \tg{d}(\Lambda)$ term in $\tg{r}$, which does not allow for the examples we consider here. We thank V. Bene\v{s} and M. Petr\'akov\'a for pointing this out. The proof was still correct, but the examples satisfied the assumption only for $\Lambda$ being a centred ball.} Equivalently, the limit in \eqref{eq:locenergy} is already attained at the smallest $n\geq 1$ such that $\Lambda_n\supset \Lambda\oplus B(0,\tg{r})$.
  %\textcolor{blue}{Nelle notazioni di prima, $\Lambda\oplus B(0,\tg{r}(\gamma,\Lambda))$, con $\tg{r}(\gamma,\Lambda)$ il pi\`u piccolo valore per cui $B(0,2 \tg{d}(\Lambda) + 2 \,\tg{l}(\tg{t}) + 2\, \tg{m}(\gamma_\Lambda) + 1)\subset \Lambda\oplus B(0,k)$.}
   %Equivalently, $\tg{r}(\gamma,\Lambda)$ is the smallest $r>0$ such that $B(0,2 \tg{d}(\Lambda) + 2 \,\tg{l}(\tg{t}) + 2\, \tg{m}(\gamma_\Lambda) + 1)\subset \Lambda\oplus B(0,r)$. 
  \item[$(\mathcal{H}_{loc.st})$] Fix $\Lambda\in\mathcal{B}_b(\IR^d)$. For any $\tg{t}\geq1$ there exists a constant $\tg{c}' = \tg{c}'(\Lambda,\tg{t}) \geq 0$ such that the following \emph{stability} of the conditional energy holds, uniformly for all $\xi\in\M^{\tg{t}}$:
  \begin{equation}\label{H:locstab}
    H_\Lambda(\gamma_\Lambda\xi_{\Lambda^c}) \geq -\tg{c}'\langle\gamma_\Lambda,1+\norm{m}^{d+\delta}\rangle,\quad \gamma_\Lambda\in\M_\Lambda.
  \end{equation} \vspace{2mm}
\end{description}

\begin{Remarks}
	\begin{enumerate}[label=\emph{\roman*.}]
		\item Notice how the stability assumption $(\mathcal{H}_{st})$ is weaker than the usual Ruelle stability $H(\gamma)\geq - \tg{c}\lvert\gamma\rvert = -\tg{c}\langle \gamma, 1\rangle$, for the presence of the mark-dependent negative term $-\tg{c}\langle \gamma, \norm{m}^{d+\delta}\rangle$.
		\item Assumption \eqref{H:range} has the following interpretation: there is no influence from the points of $\gamma_{(\Lambda\oplus B(0,\tg{r}))^c}$ on the points of $\gamma_{\Lambda}$: \mbox{$H_\Lambda(\gamma) = H_\Lambda(\gamma_{\Lambda\oplus B(0,\tg{r})})$}. The form of the range $\tg{r}$ models the case where two points $\msb{x}=(x,m),\msb{y}=(y,n)\in\mathcal{E}$ of a configuration $\gamma$ are not in interaction whenever $B(x,\norm{m})\cap B(y,\rVert n\rVert)=\emptyset$.
		
		It is easy to see from Lemma \ref{lemma:range} that $H_\Lambda(\gamma) = H_\Lambda(\gamma_{B(0,2\tg{d}(\Lambda)+2\tg{l}(\tg{t})+2\tg{m}(\gamma_\Lambda)+1)})$. Therefore, the \emph{range} of the energy $H_\Lambda$ at the configuration $\gamma$ is smaller  than $\tg{r}(\gamma,\Lambda)$, which is \emph{finite but random} since it depends on $\gamma$. This range may not be uniformly bounded when $\gamma$ varies.	\end{enumerate}
	\vspace{2mm}
\end{Remarks}

\begin{Lemma} \label{lemma:partfunctionfinite}
  Under assumptions $(\mathcal{H}_{st})$ and $(\mathcal{H}_{m})$ the partition function $Z_\Lambda$ is well defined, that is, it is finite and positive.
\end{Lemma}
\begin{proof}
  \begin{align*}
    Z_\Lambda&\geq\pi^z_\Lambda(\underline{o}) = e^{-z\lvert\Lambda\rvert}>0;\\
  \begin{split}
    Z_\Lambda &= \int e^{-H(\gamma_\Lambda)}\pi^z_\Lambda(d\gamma) \overset{\eqref{H:stab}}{\leq} \int e^{\tg{c}\langle \gamma_{\Lambda},1+\norm{m}^{d+\delta}\rangle}\pi^z_\Lambda(d\gamma)\\
    &\leq e^{-z\lvert\Lambda\rvert}\exp\big\{e^{\tg{c}}z\lvert\Lambda\rvert\int_{\IR_+}e^{\tg{c}\marknorm^{d+\delta}}\rho(d\marknorm)\big\}\overset{\eqref{H:mark}}{<}+\infty.
  \end{split}
  \end{align*}
\end{proof}
\vspace{1em}

We provide here examples of energy functionals on marked configurations, which satisfy the assumptions above. Section \ref{section:diffusion} provides, in the context of interacting diffusions, a further example of a pair interaction that acts on both locations and marks of a configuration, where the mark space is a path space.
\begin{Example}[Geometric multi-body interaction in $\IR^2$]\label{ex:quermass} Consider the marked-point state space $\mathcal{E}=\IR^2\times\IR_+$, and recall that one can associate, to any finite configuration $\gamma=\{(x_1,m_1),\dots,(x_N,m_N)\}$, $N\geq 1$, the germ-grain set
	\begin{equation*}
  		\Gamma	= \bigcup_{i=1}^N B(x_i,m_i)\subset\IR^2.
  	\end{equation*}
  	Consider, as reference mark measure, a measure $R$ on $\IR_+$ satisfying $(\mathcal{H}_m)$, that is, there exists $\delta>0$ such that
  	\begin{equation*}
  		\int_{\IR_+}e^{\ell^{2+2\delta}}R(d\ell)<+\infty.
  	\end{equation*}
  	The \emph{Quermass energy functional} $H^Q$ (see \cite{kendall_van_lieshout_baddeley_1999}) is defined as any linear combination of area, perimeter, and Euler-Poincar\'{e} characteristic functionals:
  	\begin{equation*}
    	H^Q(\gamma) = \alpha_1 \Area(\Gamma) + \alpha_2 \Per(\Gamma) + \alpha_3 \chi(\Gamma), \ \alpha_1,\alpha_2,\alpha_3\in\IR.
   \end{equation*}
   Notice how this interaction, depending on the values of the parameters $\alpha_i$, can be attractive or repulsive. It is difficult (and not useful) to decompose this multi-body energy functional as the sum of several $k$-body interactions. The functional $H^Q$ satisfies assumptions $(\mathcal{H}_{st})$, $(\mathcal{H}_{r})$, and $(\mathcal{H}_{loc.st})$. Indeed, it even satisfies the following stronger conditions:
 	\begin{itemize}
   		\item there exists a constant $\tg{c}$ such that, for any finite configuration $\gamma$,
    	\begin{equation*}\tag{\emph{two-sided stability}}
      		\lvert H^Q(\gamma)\rvert\leq \tg{c}\langle\gamma,1+\norm{m}^{2}\rangle.
    	\end{equation*}
    	\item For any $\Lambda\in\mathcal{B}_b(\IR^2)$ and $\tg{t}\geq 1$, there exists $\tg{c}'(\Lambda,\tg{t})$ such that, for any $\gamma\in\M$, $\xi\in\M^{\tg{t}}$,
    	\begin{equation*}\tag{\emph{two-sided loc. stability}}
      		\lvert H^Q_\Lambda(\gamma_\Lambda\xi_{\Lambda^c})\rvert \leq \tg{c}'(\Lambda,\tg{t})\langle\gamma_\Lambda,1+\norm{m}^{2}\rangle.
    	\end{equation*}
  	\end{itemize}
	Under these stronger conditions than ours, the existence for the Quermass-interaction model was proved in \cite{dereudre_2009}; notice that $H^Q$ is not superstable.

	For more examples of geometric interactions, see \cite{dereudre_2019}.
\end{Example}

\begin{Example}[Two-body interactions]\label{ex:pair}
\begin{description}[font=\emph]
	\item[i. Interacting hard spheres of random radii] On $\mathcal{E}=\IR^d\times\IR_+$, consider a model of hard balls centred at points $x_i$, of random radii $m_i$ distributed according to a measure $R$ satisfying Assumption $(\mathcal{H}_m)$. The hard-core energy functional of a finite configuration $\gamma = \{(x_1,m_1),\dots(x_N,m_N)\}$, $N\geq 1$, is given by 
	\begin{equation*}
		H(\gamma)=\sum_{1\leq i<j\leq N}(+\infty)\ \1_{\{B(x_i,m_i)\cap B(x_j,m_j)\neq\emptyset\}},
	\end{equation*}
	with the convention $+\infty\cdot 0 =0$.
	\item[ii. Non-negative pair interaction] On $\mathcal{E}=\IR^d\times\mathcal{S}$, consider any energy functional $H$ of the form $H(\gamma)=\sum_{1\leq i<j\leq N}\Phi(\msb{x}_i,\msb{x}_j)$, where
\begin{equation*}
	\Phi(\msb{x}_i,\msb{x}_j) = \phi(\lvert x_i-x_j\rvert)\ \1_{\{\lvert x_i-x_j\rvert\leq\norm{m_i}+\norm{m_j}\}},
\end{equation*}
where $\phi$ is non-negative and null at $0$.
\end{description}
In both cases, since $H$ is a non-negative functional, it satisfies $(\mathcal{H}_{st})$ and $(\mathcal{H}_{loc.st})$. It is also easy to see that, by construction, the range assumption $(\mathcal{H}_r)$ also holds.
\end{Example}

\subsection{Local topology}

We endow the space of point measures with the topology of local convergence (see \cite{georgii_2011}, \cite{georgii_zessin_1993}), defined as the weak* topology induced by a class of functionals on $\M$ which we now introduce.
\begin{Definition}\label{def:tame}
  A functional $F$ is called \emph{tame} if there exists a constant $c>0$ such that
  \begin{equation*}
    \lvert F(\gamma)\rvert\leq c \ \big(1+\langle\gamma, 1+\norm{m}^{d+\delta}\rangle\big), \quad \gamma\in\M.
  \end{equation*}

We denote by $\mathcal{L}$ the set of all tame and local functionals. The topology $\tau_\mathcal{L}$ of \emph{local convergence} on $\mathcal{P}(\M)$ is then defined as the weak* topology induced by $\mathcal{L}$, i.e. the smallest topology on $\mathcal{P}(\M)$ under which all the mappings $P\mapsto\int F \,dP$, \mbox{$F\in\mathcal{L}$}, are continuous.
\end{Definition}

\section{Construction of an infinite-volume Gibbs measure}

Let us first precise the terminology (see \cite{georgii_1979}).
\begin{Definition}
  Let $H$ be an energy functional satisfying the three assumptions $(\mathcal{H}_{st})$, $(\mathcal{H}_{r})$, and $(\mathcal{H}_{loc.st})$.
  We say that a probability measure $P$ on $\M$ is an \emph{infinite-volume Gibbs measure with energy functional} $H$ if, for every finite volume $\Lambda\subset\IR^d$ and for any measurable, bounded and local functional $F\colon\M\rightarrow\IR$, the following identity (called \emph{DLR equation} after Dobrushin--Lanford--Ruelle) holds under $P$:
  \begin{equation}\label{DLR}\tag*{(DLR)$_\Lambda$}
    \int_{\M} F(\gamma)\, P(d\gamma) =
\int_{\M}\int_{\M_\Lambda}F(\gamma_\Lambda\xi_{\Lambda^c})\ \Xi_\Lambda(\xi,d\gamma)\; P(d\xi),
  \end{equation}
  where $\Xi_\Lambda$, called the \emph{Gibbsian probability kernel associated to $H$}, is defined on $\mathcal{M}_\Lambda$ by
  \begin{equation}\label{def:XiLambda}
    \Xi_\Lambda(\xi,d\gamma)
\defeq \dfrac{e^{-H_\Lambda(\gamma_\Lambda\xi_{\Lambda^c})}}{Z_\Lambda(\xi)}\pi^z_\Lambda(d\gamma),
  \end{equation}
  where $Z_\Lambda(\xi)\defeq\int_{\M_\Lambda}e^{-H_\Lambda(\gamma_\Lambda\xi_{\Lambda^c})}\pi^z_\Lambda(d\gamma)$.
\end{Definition}

\begin{Remarks}
  \begin{enumerate}
    \item The probability kernel $\Xi_\Lambda(\xi,\cdot)$ is not necessarily  well-defined for any $\xi \in \M$. In Lemma \ref{lemma:kernel}, we will show that this is the case when we restrict it to the subspace $\M^{\temp}$.
    \item The map $\xi\mapsto\Xi_\Lambda(\xi,d\gamma)$ is a priori \emph{not local} since
$\xi\mapsto H_\Lambda(\gamma_\Lambda\xi_{\Lambda^c})$ may depend on the full configuration $\xi_{\Lambda^c}$.
    \item The renormalisation factor  $Z_\Lambda(\xi)$ -- when it exists -- only depends on the external configuration $\xi_{\Lambda^c}$. Therefore $\Xi_\Lambda(\xi,\cdot)\equiv \Xi_\Lambda(\xi_{\Lambda^c},\cdot)$.
  \end{enumerate}
\end{Remarks}
We can now state the main result of this paper:
\begin{Theorem}\label{thm:Gibbs}
  Under assumptions $(\mathcal{H}_{m})$, $(\mathcal{H}_{st})$, $(\mathcal{H}_{r})$, and $(\mathcal{H}_{loc.st})$, there exists at least one infinite-volume Gibbs measure with energy functional $H$.
\end{Theorem}

\noindent
This section will have the following structure.
\begin{enumerate}[label=3.\arabic*,leftmargin=*]
  \item We define a sequence of stationarised finite-volume Gibbs measures $(\bar{P}_n)_n$.
  \item We use uniform bounds on the entropy to show the convergence, up to a subsequence, to an infinite-volume measure $\bar{P}$.
  \item We prove, using an ergodic property, that $\bar{P}$ carries only the space of tempered configurations.
  \item Noticing that, for any fixed $\Lambda\in\mathcal{B}_b(\IR^d)$, $\bar{P}_n$ does not satisfy \ref{DLR}, we introduce a new sequence $(\hat{P}_n)_n$ asymptotically equivalent to $(\bar{P}_n)_n$ but satisfying \ref{DLR}. We use appropriate approximations, by localising the interaction, to show that also $\bar{P}$ satisfies \ref{DLR}.
\end{enumerate}

\subsection{A stationarised sequence}\label{section:stationary}

In this subsection, we extend each finite-volume measure $P_n \defeq P_{\Lambda_n}$, $\Lambda_n = [-n,n)^d$, defined on $\M_{\Lambda_n}$ to a probability measure $\bar{P}_n$ on the full space $\M$, invariant under lattice-translations.
\smallbreak

We start with the following
\begin{Lemma}\label{lemma:bounds}
	There exists a constant $a_1$ such that
	\begin{equation}\label{eq:entr0}
		\forall n\geq1, \quad J_n \defeq \int_{\M}\langle\gamma,1+\norm{m}^{^{d+\delta}}\rangle P_n(d\gamma)\leq a_1 \lvert\Lambda_n\rvert.
	\end{equation}
	\end{Lemma}
	\begin{proof}
    We partition the space of configurations $\M_{\Lambda_n}$ in three sets:
    \begin{description}
      \item $\M_{\Lambda_n}^{(1)}\defeq\{\gamma\in\M_{\Lambda_n}:\langle\gamma,1+\norm{m}^{d+\delta}\rangle\leq a_{11}\lvert\Lambda_n\rvert\}$,
      \item $\M_{\Lambda_n}^{(2)}\defeq\{\gamma\in\M_{\Lambda_n}:\langle\gamma,1+\norm{m}^{d+\delta}\rangle> a_{11}\lvert\Lambda_n\rvert,\ \lvert\gamma\rvert>a_{12}\lvert\Lambda_n\rvert\}$,
      \item $\M_{\Lambda_n}^{(3)}\defeq\{\gamma\in\M_{\Lambda_n}:\langle\gamma,1+\norm{m}^{d+\delta}\rangle> a_{11}\lvert\Lambda_n\rvert,\ \lvert\gamma\rvert\leq a_{12}\lvert\Lambda_n\rvert\}$,
    \end{description}
    for some constants $a_{11},a_{12}$ which will be fixed later.
Therefore, the integral $J_n$ can be written as the sum of three integrals, $J_n^{(1)},J_n^{(2)},J_n^{(3)}$, resp. over each of these sets.

    \smallbreak

    The first term is straightforward:
    \begin{equation*}
      J_n^{(1)} \defeq \int_{\M_{\Lambda_n}^{(1)}}\langle\gamma_{\Lambda_n},1+\norm{m}^{d+\delta}\rangle P_n(d\gamma)\leq a_{11}\lvert\Lambda_n\rvert.
    \end{equation*}

    For the second term,
    \begin{equation*}
    \begin{split}
        &J_n^{(2)} \defeq \int_{\M_{\Lambda_n}^{(2)}}\langle\gamma_{\Lambda_n},1+\norm{m}^{d+\delta}\rangle P_n(d\gamma) \\
              &\overset{\eqref{eq:finiteGibbs}}{\leq} \int_{\M_{\Lambda_n}}\1_{\{\lvert\gamma_{\Lambda_n}\rvert>a_{12}\lvert\Lambda_n\rvert\}}\langle\gamma,1+\norm{m}^{d+\delta}\rangle \dfrac{1}{Z_{\Lambda_n}}e^{-H(\gamma_{\Lambda_n})}\pi^z_{\Lambda_n}(d\gamma) \\
        &\overset{\mathclap{\eqref{H:stab}}}{\leq}\dfrac{e^{-z\lvert\Lambda_n\rvert}}{Z_{\Lambda_n}}\sum_{k=a_{12}\lvert\Lambda_n\rvert}^{+\infty}\dfrac{(z\lvert\Lambda_n\rvert)^k}{k!}\int_{\mathcal{S}^k} e^{\tg{c}\sum_{i=1}^{k}(1+\norm{m_i}^{d+\delta})} \sum_{j=1}^{k}(1+\norm{m_j}^{d+\delta})R(dm_1)\dots R(dm_k) \\
      &\leq\sum_{k=a_{12}\lvert\Lambda_n\rvert}^{+\infty}\dfrac{(z\lvert\Lambda_n\rvert)^k}{k!}\ k\ \left(\int(1+\marknorm^{d+\delta})e^{\tg{c}(1+\marknorm^{d+\delta})}\rho(d\marknorm)\right)\left(\int e^{\tg{c}(1+\marknorm^{d+\delta})}\rho(d\marknorm)\right)^{k-1}.
    \end{split}
    \end{equation*}
    Using \eqref{H:mark}, we are able to find a constant $b_1$ such that
    \begin{equation*}
      \int(1+\marknorm^{d+\delta})\, e^{\tg{c}(1+\marknorm^{d+\delta})}\rho(d\marknorm)\leq b_1.
    \end{equation*}
    We then get
    \begin{equation*}
      J_n^{(2)}\leq\sum_{k=a_{12}\lvert\Lambda_n\rvert}^{+\infty}\dfrac{(zb_1\lvert\Lambda_n\rvert)^k}{k!}k\leq \sum_{k=a_{12}\lvert\Lambda_n\rvert}^{+\infty}\dfrac{(2zb_1\lvert\Lambda_n\rvert)^k}{k!}\leq e^{2z b_1\lvert\Lambda_n\rvert}\IP\Big(S_{\lvert\Lambda_n\rvert}\geq a_{12}\lvert\Lambda_n\rvert\Big),
    \end{equation*}
    for a sequence $(S_m)_{m\geq1}$ of Poisson random variables with parameter $2z b_1 m$.\\
    Recalling the Cram\'{e}r-Chernoff inequality (cf. \cite{cramer_1938})
    \begin{equation*}
      \IP(\tfrac{1}{m}S_m\geq a_{12})\leq e^{-mL^*(a_{12})},
    \end{equation*}
    where $L^*(x) = 2zb_1 + x\log x - x(1+\log(2zb_1))$ is the Legendre transform associated to the Poisson random variable of parameter $2zb_1$, we can choose $a_{12}$ large enough, so that $\log a_{12}\geq 1+\log(2zb_1)$. Thus \mbox{$L^*(a_{12})\geq2zb_1$}, and we get that $J_{n}^{(2)}\leq1$.

    For the third term,
    \begin{equation*}
    \begin{split}
      J_n^{(3)}&\overset{\eqref{eq:finiteGibbs}}{=} \int_{\M_{\Lambda_n}}\1_{\{\langle\gamma,1+\norm{m}^{d+\delta}\rangle > a_{11}\lvert\Lambda_n\rvert,\ \lvert\gamma\rvert\leq a_{12}\lvert\Lambda_n\rvert\}}\langle\gamma,1+\norm{m}^{d+\delta}\rangle \dfrac{1}{Z_{\Lambda_n}}e^{-H(\gamma_{\Lambda_n})}\pi^z_{\Lambda_n}(d\gamma) \\
      &\leq \dfrac{e^{-z\lvert\Lambda_n\rvert}}{Z_{\Lambda_n}} \sum_{k=0}^{a_{12}\lvert\Lambda_n\rvert}\dfrac{(z\lvert\Lambda_n\rvert)^k}{k!}\int_{\mathcal{S}^k}\1_{\{ \sum_{i=1}^{k}(1+\norm{m_i}^{d+\delta}) > a_{11}\lvert\Lambda_n\rvert\}} \\
      &\phantom{\leq\sum_{k=0}^{A_4\lvert\Lambda_n\rvert}\dfrac{(z\Lambda_n)^k}{k!}\int\dots\int\1_{\{\sum_{(x,m)}}} e^{\tg{c}\sum_{i=1}^{k}(1+\norm{m_i}^{d+\delta})} \sum_{j=1}^{k}(1 + \norm{m_j}^{d+\delta})R(dm_1)\dots R(dm_k) \\
      &\leq \sum_{k=0}^{a_{12}\lvert\Lambda_n\rvert}\dfrac{(z\lvert\Lambda_n\rvert)^k}{k!}\int_{\IR^k_+}\1_{\{ \sum_{i=1}^{k}(1+\marknorm_i^{d+\delta}) > a_{11}\lvert\Lambda_n\rvert\}} \\
      &\phantom{\leq\sum_{k=0}^{A_4\lvert\Lambda_n\rvert}\dfrac{(z\Lambda_n)^k}{k!}\int\dots\int\1_{\{\sum_{(x,m)}}} e^{\tg{c}\sum_{i=1}^{k}(1+\marknorm_i^{d+\delta})} \sum_{j=1}^{k}(1 + \marknorm_j^{d+\delta})\rho(d\marknorm_1)\dots \rho(d\marknorm_k).
    \end{split}
    \end{equation*}

    Applying Cauchy-Schwarz inequality, we find:
    \begin{equation*}
    \begin{split}
      J_n^{(3)}\leq& \sum_{k=0}^{a_{12}\lvert\Lambda_n\rvert}\dfrac{(z\lvert\Lambda_n\rvert)^k}{k!}\sqrt{\rho^{\otimes k}\Big(\sum_{i=1}^{k}\big(1+\marknorm_i^{d+\delta}\big)>a_{11}\lvert\Lambda_n\rvert\Big)}\\
      &\sqrt{\int_{\IR^k_+} e^{2\tg{c}\sum_{i=1}^{k}(1+\marknorm_i^{d+\delta})}\Big(\sum_{j=1}^{k}(1+\marknorm_j^{d+\delta})\Big)^2\rho(d\marknorm_1)\dots \rho(d\marknorm_k)} \\
      \leq& \sqrt{\rho^{\otimes a_{12}\lvert\Lambda_n\rvert}\Big(\sum_{i=1}^{a_{12}\lvert\Lambda_n\rvert}\big(1+\marknorm_i^{d+\delta}\big)>a_{11}\lvert\Lambda_n\rvert\Big)} \\
		&\sum_{k=0}^{a_{12}\lvert\Lambda_n\rvert}\dfrac{(z\lvert\Lambda_n\rvert)^k}{k!}\sqrt{k^2\int(1+\marknorm^{d+\delta})^2e^{2\tg{c}(1+\marknorm^{d+\delta})}\rho(d\marknorm)\left(\int e^{2\tg{c}(1+\marknorm^d)}\rho(d\marknorm)\right)^{k-1}}.
   \end{split}
	\end{equation*}
	Using \eqref{H:mark}, there exists a positive constant $b_2$ such that
	\begin{equation*}
		\int (1+\marknorm^{d+\delta})^2e^{2\tg{c}(1+\marknorm^{d+\delta})}\rho(d\marknorm) \leq b_2.
	\end{equation*}
	Thus
	\begin{equation*}
	\begin{split}
		J_n^{(3)}&\leq \sqrt{\rho^{\otimes a_{12}\lvert\Lambda_n\rvert}\Big(\sum_{i=1}^{a_{12}\lvert\Lambda_n\rvert}\big(1+\marknorm_i^{d+\delta}\big)>a_{11}\lvert\Lambda_n\rvert\Big)} \sum_{k=0}^{a_{11}\lvert\Lambda_n\rvert}\dfrac{(z\sqrt{b_2}\lvert\Lambda_n\rvert)^k}{k!}k \\
		&\leq \sqrt{\rho^{\otimes a_{12}\lvert\Lambda_n\rvert}\Big(\sum_{i=1}^{a_{12}\lvert\Lambda_n\rvert}\big(1+\marknorm_i^{d+\delta}\big)>a_{11}\lvert\Lambda_n\rvert\Big)}\ e^{2z\sqrt{b_2}\lvert\Lambda_n\rvert}.
   	\end{split}
	\end{equation*}
	Using again the Cram\'{e}r-Chernoff inequality, we can choose $a_{11}$ large enough such that $\bar{L}^*$, the Legendre transform of the image measure of $\rho$ by $\marknorm\mapsto 1+\marknorm^{d+\delta}$, satisfies $\bar{L}^*(a_{11})\geq 4z\sqrt{b_2}$ (since it is stricly increasing on the positive half-line). Thus
	\begin{equation*}
   		\rho^{\otimes a_{12}\lvert\Lambda_n\rvert}\Big(\sum_{i=1}^{a_{12}\lvert\Lambda_n\rvert}\big(1+\marknorm_i^{d+\delta}\big)>a_{11}\lvert\Lambda_n\rvert\Big)\leq e^{-4z\sqrt{b_2}\lvert\Lambda_n\rvert},
   \end{equation*}
   	which yields $J_n^{(3)}\leq1$.\\
	Putting it all together, the claim of Lemma \ref{lemma:bounds} follows with $a_1\defeq a_{11}+2$.
	\end{proof}

We start by considering the probability measure $\tilde{P}_n$ on $\M$, under which the configurations in the disjoint blocks $\Lambda_n^\kappa \defeq \Lambda_n + 2n\kappa$, $\kappa\in\IZ^d$, are independent, with identical distribution $P_n$. We then build the empirical field associated to the probability measure $\tilde{P}_n$, i.e. the sequence of lattice-stationarised probability measures
\begin{equation} \label{def:Pbarn}
\bar{P}_n = \dfrac{1}{(2n)^d}\sum_{\kappa\in\Lambda_n\cap\IZ^d}\tilde{P}_n\circ\theta_\kappa^{-1},
\end{equation}
where $\theta_\kappa$ is the translation on $\IR^d$ by the vector $\kappa \in\IZ^d$.
\begin{Remarks}
  \begin{enumerate}
    \item As usual we identify the translation $\theta_\kappa$ on $\IR^d$ with the image of a point measure under such translation.
    \item So constructed, the probability measure $\bar{P}_n$ is invariant under $(\theta_\kappa)_{\kappa\in\IZ^d}$.
    \item An upper bound similar to \eqref{eq:entr0} holds also under $\bar{P}_n$:
    \begin{equation*}
      \exists a_2>0, \forall n\geq1,\quad \int_{\M}\langle\gamma_{\Lambda_n},1+\norm{m}^{^{d+\delta}}\rangle \bar{P}_n(d\gamma)\leq a_2 \, \lvert\Lambda_n\rvert.
    \end{equation*}
    Moreover, using stationarity and the fact that the covering $\Lambda_n = \bigcup_\kappa\Lambda_1^\kappa$ contains $n^d$ terms,
    \begin{equation}\label{eq:entrbarP}
    \begin{split}
  &\int_{\M}\langle\gamma_{\Lambda_1},1+\norm{m}^{^{d+\delta}}\rangle \bar{P}_n(d\gamma)
= \int_{\M}\dfrac{1}{n^d}\sum_\kappa \langle\gamma_{\Lambda_1^\kappa},1+\norm{m}^{^{d+\delta}}\rangle \bar{P}_n(d\gamma)\\
      &=\dfrac{1}{n^d}\int_{\M}\langle\gamma_{\Lambda_n},1+\norm{m}^{d+\delta}\rangle\bar{P}_n(d\gamma)\leq \dfrac{1}{n^d} (2n)^d a_2 = 2^d a_2.
    \end{split}
    \end{equation}
  \end{enumerate}
\end{Remarks}
\vspace{8mm}

As we will see in the following subsection, in order to prove that $(\bar{P}_n)_n$ admits an accumulation point, it is enough to prove that all elements of the sequence belong to the same entropy level set.

\subsection{Entropy bounds}\label{section:entropy}

Let us now introduce the main tool of our study, the specific entropy of a (stationary) probability measure on $\M$.
  \vspace{5mm}

\begin{Definition}
  Given two probability measures $Q$ and $Q'$ on $\M$, and any finite-volume $\Lambda\subset\IR^d$, the \emph{relative entropy} of $Q'$ with respect to $Q$ on $\Lambda$ is defined as
  \begin{equation*}
    I_{\Lambda}(Q\vert Q')\defeq\begin{cases*}
      \begin{aligned}
      &\int \log f\ dQ_{\Lambda} &\qquad\text{ if } Q_{\Lambda}\preccurlyeq Q'_{\Lambda} \text{ with } f\defeq\tfrac{dQ_{\Lambda}}{dQ'_{\Lambda}},\\
      &+\infty &\qquad\text{otherwise},
      \end{aligned}
    \end{cases*}
  \end{equation*}
  where $Q_\Lambda$ (resp. $Q'_\Lambda$) is the image of $Q$ (resp. $Q'$) under the mapping $\gamma\mapsto\gamma_\Lambda$.
\end{Definition}
As usual,
\begin{Definition}
  The \emph{specific entropy} of $Q$ with respect to $Q'$ is defined by
  \begin{equation*}
    \mathscr{I}(Q\vert Q') = \lim\limits_{n\rightarrow+\infty}\dfrac{1}{\lvert\Lambda_n\rvert}I_{\Lambda_n}(Q\vert Q').
  \end{equation*}
\end{Definition}
From now on, the reference measure $Q'$ will be the marked Poisson point process $\pi^z$ with intensity measure $z\, dx\otimes R(dm)$. In this case, the specific entropy of a probability measure $Q$ with respect to $\pi^z$ is always well defined if $Q$ is stationary under the lattice translations $(\theta_\kappa)_{\kappa\in\IZ^d}$. Moreover, recall that for any $a>0$, the $a$-entropy level set
\begin{equation*}
  \mathcal{P}(\M)_{\leq a} \defeq \Big\{Q\in\mathcal{P}(\M), \text{ stationary under } (\theta_\kappa)_{\kappa\in\IZ^d}: \mathscr{I}(Q\vert\pi^z)\leq a\Big\}
\end{equation*}
is relatively compact for the topology $\tau_\mathcal{L}$, as proved in \cite{georgii_zessin_1993}.

\begin{Proposition}\label{prop:entropy}
  There exists a constant $a_3>0$ such that,
  \begin{equation*}
  \forall n \geq 1, \quad	\mathscr{I}(\bar{P}_n\vert\pi^z)\leq a_3
  \end{equation*}
where $\bar{P}_n\in\mathcal{P}(\M)$ is the empirical field defined by \eqref{def:Pbarn}.
\end{Proposition}
\begin{proof}
  Since the map $Q\mapsto\mathscr{I}(Q\vert\pi^z)$ is affine, it holds
  \begin{equation*}
  \begin{split}
    \mathscr{I}(\bar{P}_n\vert\pi^z) &= \dfrac{1}{(2n)^d}\sum_{\kappa\in\IZ^d\cap\Lambda_n}\mathscr{I}(\tilde{P}_n\circ\theta_\kappa^{-1}\vert\pi^z) \\
    &= \mathscr{I}(\tilde{P}_n\vert\pi^z) =  \lim\limits_{m\rightarrow+\infty}\dfrac{1}{\lvert 2m\Lambda_n\rvert}I_{2m\Lambda_n}(\tilde{P}_n\vert\pi^z) \\
    &= \lim\limits_{m\rightarrow+\infty}\dfrac{1}{(2m)^d\lvert\Lambda_n\rvert}(2m)^d\ I_{\Lambda_n}(P_n\vert\pi^z) = \dfrac{1}{\lvert\Lambda_n\rvert}I_{\Lambda_n}(P_n\vert\pi^z).
  \end{split}
  \end{equation*}
  Using the stability of the energy functional, we find
  \begin{equation*}
  \begin{split}
  I_{\Lambda_n}(P_n\vert\pi^z) = -\int H(\gamma)P_n(d\gamma)-\log(Z_{\Lambda_n})\overset{\eqref{H:stab}}{\leq}\tg{c}\int\langle\gamma,1+\norm{m}^{d+\delta}\rangle P_n(d\gamma) + z\lvert\Lambda_n\rvert.
  \end{split}
  \end{equation*}
  From Lemma \ref{lemma:bounds}, we know that inequality \eqref{eq:entr0} holds.
  Defining $a_3 \defeq \tg{c}a_1+z$, we conclude that, uniformly in $n\geq1$, $\mathscr{I}(\bar{P}_n\vert\pi^z)\leq a_3$.
\end{proof}

From the above proposition we deduce that the sequence $(\bar{P}_n)_{n\geq 1}$ belongs to the  relatively compact set $\mathcal{P}(\M) _{\leq a_3} $. It then admits at least one converging subsequence which we will still denote by $(\bar{P}_n)_{n\geq 1}$ for simplicity. The limit measure, here denoted by $\bar{P}$, is stationary under the translations $(\theta_\kappa)_{\kappa\in\IZ^d}$. We will prove in what follows that $\bar{P}$ is the infinite-volume Gibbs measure we are looking for.

\subsection{Support of the infinite-volume limit measure}\label{section:support}

We now justify the introduction of a set of tempered configurations as the right support of each of the probability measures $\bar{P}_n$, $n\geq 1$, as well as of the constructed limit probability measure $\bar{P}$.
\begin{Proposition}\label{prop:supbarP}
	The measures $\bar{P}_n$, $n\geq 1$, and the limit measure $\bar{P}$ are all supported on the tempered configurations, i.e.
	\begin{equation*}
		\forall n\geq 1,\quad \bar{P}_n\big(\M^{\temp}\big) = \bar{P}\big(\M^{\temp}\big) = 1.
	\end{equation*}
\end{Proposition}

\begin{proof}
  Let us show that, for $\bar{P}$ (resp. $\bar{P}_n$)-a.e. $\gamma\in\M$, there exists $\tg{t}=\tg{t}(\gamma)\geq1$ such that
  \begin{equation}\label{eq:temp3}
    \sup\limits_{l\in\IN^*}\dfrac{1}{l^d}\langle\gamma_{B(0,l)},1+\norm{m}^{d+\delta}\rangle\leq\tg{t}.
  \end{equation}
  From \eqref{eq:entrbarP}, we know that
  \begin{equation}\label{eq:supp}
    \forall n\geq1, \quad \int\langle\gamma_{[-1,1)^d},1+\norm{m}^{d+\delta}\rangle\bar{P}_n(d\gamma)\leq 2^d a_2.
  \end{equation}
  Since the integrand is a tame local function, the same inequality remains true when passing to the limit:
  \begin{equation*}
    \int\langle\gamma_{[-1,1)^d},1+\norm{m}^{d+\delta}\rangle\bar{P}(d\gamma)\leq 2^d a_2.
  \end{equation*}
  The integrability of $\langle\gamma_{[-1,1)^d},1+\norm{m}^{d+\delta}\rangle$ under $\bar{P}$ (resp. $\bar{P}_n$) is precisely what we need in order to apply the ergodic theorem in \cite{nguyen_zessin_1979}. Doing so yields the following spatial asymptotics, where we have $\bar{P}$ (resp. $\bar{P}_n$)-a.s. convergence to the conditional expectation under $\bar{P}$ (resp. $\bar{P}_n$) with respect to the $\sigma$-field $\mathcal{J}$ of $(\theta_\kappa)_{\kappa\in\IZ^d}$- invariant sets:
  \begin{equation*}
    \lim\limits_{l\rightarrow+\infty}\dfrac{1}{\lvert B(0,l)\rvert}\langle\gamma_{B(0,l)},1+\norm{m}^{d+\delta}\rangle
= \frac{1}{2^d}\IE_{\bar{P}}\Big[\langle\gamma_{[-1,1)^d},1+\norm{m}^{d+\delta}\rangle\ \vert\ \mathcal{J}\Big]
  \end{equation*}
  (resp. $\IE_{\bar{P}_n}$). This implies that, $\bar{P}$ (resp. $\bar{P}_n$)-a.s.,
  \begin{equation*}
    \lim\limits_{l\rightarrow+\infty}\dfrac{1}{\lvert B(0,l)\rvert}\langle\gamma_{B(0,l)},1+\norm{m}^{d+\delta}\rangle <+\infty
  \end{equation*}
  so that, a fortiori, \eqref{eq:temp3} holds, and the proposition is proved.
\end{proof}

In Subsection \ref{section:gibbs}, in order to prove Gibbsianity of the limit measure, we need more: a uniform estimate of the support of the measures $\bar{P}_n$, $n\geq 1$. For this reason, we introduce the increasing family $(\underline{\M}^l)_{l\in\IN^*}$ of subsets of $\M^\temp$, defined by
\begin{equation*}
	\underline{\M}^l\defeq \big\{\gamma\in\M^\temp:\, \forall k\in\IN^*,\, k \geq l,\, \forall (x,m)\in\gamma_{B(0,2k+1)^c},\  B(x,\norm{m})\cap B(0,k)=\emptyset \big\}.
\end{equation*}
Notice that, thanks to Lemma \ref{lemma:range}, for any $\tg{t}\geq 1$, $\M^\tg{t}\subset\underline{\M}^{\tg{l}(\tg{t})}$ (see Figure \ref{fig:circles}).

\begin{Proposition}\label{prop:supbarPn}
For any $\epsilon>0$, there exists $l\geq 1$ such that
	\begin{equation*}
		\forall n\geq 1,\quad \bar{P}_n(\underline{\M}^l)\geq 1-\epsilon.
	\end{equation*}
\end{Proposition}

\begin{proof}
	We want to find $l\geq 1$ such that 
	\begin{equation}\label{eq:estunderM}
		\bar{P}_n\Big(\sup_{k\geq l}\sup_{(x,m)\in\gamma_{B(0,2k+1)^c}}\frac{\norm{m}}{\abs{x}}\geq \frac{1}{2}\Big)\leq \epsilon.
	\end{equation}
For any $\kappa=(\kappa^1,\dots,\kappa^d)\in\IZ^d$, let $D_\kappa=[\kappa^1,\kappa^1+1)\times\dots\times[\kappa^d,\kappa^d+1)\subset\IR^d$. We list all the elements of $\IZ^d$ by a sequence $(\kappa_i)_{i\in\IN}\subset\IZ^d$ that forms a spiral, starting at $0$; in particular, there exist constants $a,b>0$ (depending on the dimension $d$), such that $i\,a\leq \abs{\kappa_i}^{d}\leq i\,b$. We can then compute, for any $l\geq 1$,
	\begin{equation*}
	\begin{split}
		&\sum_{\substack{\kappa\in\IZ^d: \\ \abs{\kappa}\geq 2l}}\bar{P}_n\big(\tg{m}(\gamma_{D_\kappa})\geq \tfrac{1}{2}\abs{\kappa}\big)  = \sum_{\substack{i\geq 1: \\ \abs{\kappa_i}\geq 2l}} \bar{P}_n\big(\tg{m}(\gamma_{D_{\kappa_i}})\geq \tfrac{1}{2}\abs{\kappa_i}\big)\\
		&\leq \sum_{i\geq (2l)^d/b} \bar{P}_n\big(\tg{m}(\gamma_{D_{\kappa_i}})\geq \tfrac{1}{2}\abs{\kappa_i}\big) \leq \sum_{i\geq (2l)^d/b} \bar{P}_n\big(\tg{m}(\gamma_{D_{\kappa_i}})^d\geq \frac{a}{2^d}i\big)\\
		&\leq \sum_{i\geq (2l)^d/b} \bar{P}_n\Big(\underbrace{\frac{2^d}{a}\hspace{-.3em}\sum_{(x,m)\in\gamma_{D_{0}}}(1 + \norm{m}^d)}_{\eqdef\ F(\gamma)} \geq i\Big)\\
		&\leq \IE_{\bar{P}_n}[\1_{\{F(\gamma)\geq (2l)^d/b\}}F(\gamma)]\\
		&= \frac{2^d}{a}\IE_{\bar{P}_n}\Big[\1_{\{\sum_{(x,m)\in\gamma_{D_0}}(1+\norm{m}^d)\geq \frac{a}{b} l^d\}} \sum_{(x,m)\in\gamma_{D_0}}(1+\norm{m}^d)\Big].
	\end{split}
	\end{equation*}
	To control this expression, recall the following result (due to H.-O. Georgii and H. Zessin), which proves that point configurations in $\IR^d$ with marks in a complete, separable metric space, satisfy a local equi-integrability property on entropy level sets, with respect to the marks:
	\begin{Lemma}[\normalfont{\cite{georgii_zessin_1993}, Lemma 5.2}]\label{lemma:EI}
	\emph{For any measurable non-negative function \mbox{$f:\mathcal{S}\rightarrow\IR_+$} and for every $a>0$ and $\Delta\in\mathcal{B}_b(\IR^d)$,}
  	\begin{equation*}
      	\lim\limits_{N\rightarrow\infty}\sup\limits_{P\in\mathcal{P}(\M)_{\leq a}}\IE_{P}\Big[\1_{\{\langle\gamma_\Delta,f\rangle\geq N\}}\langle\gamma_\Delta,f\rangle\Big] = 0.
    \end{equation*}
   \end{Lemma} 
   	Applying this result to the sequence $(\bar{P}_n)_n$, with $f(x,m) = 1+\norm{m}^{d}$ and \mbox{$\Delta=D_0$}, we have that, for any $\epsilon>0$, there exists $l\geq 1$ large enough, such that
	\begin{equation*}
		\forall n\geq 1,\quad \bar{P}_n\Big( \sup_{\kappa\in\IZ^d,\ \abs{\kappa}\geq 2l}\frac{1}{\abs{\kappa}}\tg{m}(\gamma_{D_\kappa})\geq \frac{1}{2}\Big)\leq \epsilon.
	\end{equation*}
	For any $(x,m)\in\gamma_{B(0,2k+1)^c}$, with $k\geq l$, there exists $\kappa\in\IZ^d$ with $\abs{\kappa}\geq 2k$, such that $(x,m)\in\gamma_{D_\kappa}$; since then $\frac{\norm{m}}{\abs{x}}\leq \frac{\norm{m}}{\abs{\kappa}}$, we find that \eqref{eq:estunderM} holds, and the claim follows.
\end{proof}

\begin{Remark}
	One could have thought that such a uniform estimate held in $\M^\tg{t}$, for some $\tg{t}\geq 1$, but this is not the case; we thank one of the referees for pointing how this would not work. In order to have the uniform estimate, we had then to enlarge the set of tempered configurations by introducing $\underline{\M}^l$ instead.
\end{Remark}

\subsection{The limit measure is Gibbsian}\label{section:gibbs}

We are now ready to prove that the infinite-volume $\bar{P}$ measure we have constructed satisfies the Gibbsian property.

  \begin{Lemma}\label{lemma:kernel}
  Consider the Gibbsian kernel $\Xi_\Lambda$ defined by \eqref{def:XiLambda}. It satisfies:
  \begin{enumerate}[label=(\roman*)]
    \item For any $\xi\in\M^{\temp}$, $\Xi_\Lambda(\xi,d\gamma)$ is well defined: $Z_\Lambda(\xi)<+\infty$;
    \item For any $\Lambda$-local tame functional $F$ on $\M$, the map $\xi\mapsto\int_{\M_\Lambda} F(\gamma) \Xi_\Lambda(\xi,d\gamma)$ defined on $\M^{\temp}$ is measurable.
    \item The family $(\Xi_\Lambda)_{\Lambda\in\mathcal{B}_b(\IR^d)}$ satisfies a \emph{finite-volume compatibility condition}, in the sense that, for any ordered finite-volumes $\Lambda\subset\Delta$,
  \begin{equation}\label{eq:compatibility}
    \int_{\M_{\Delta\setminus\Lambda}} \Xi_\Lambda(\zeta_{\Delta\setminus \Lambda}\xi_{\Delta^c}, d\gamma_{\Lambda})\,  \Xi_\Delta(\xi_{\Delta^c}, d\zeta_{\Delta\setminus \Lambda})
= \Xi_\Delta (\xi_{\Delta^c}, d(\gamma_{\Lambda}\zeta_{\Delta\setminus \Lambda})).
  \end{equation}
  \end{enumerate}
\end{Lemma}

\begin{proof}
	\begin{enumerate}[label=\emph{(\roman*)}]
		\item We have to show that, for any $\xi\in\M^{\temp},\ 0<Z_\Lambda(\xi)<+\infty$. Lemma \ref{lemma:partfunctionfinite} dealt with the free boundary condition case, so this followed from the stability assumption \eqref{H:stab}. Since $H_\Lambda(\gamma_\Lambda\xi_{\Lambda^c})\neq H(\gamma_\Lambda)$, this now follows in the same way from \eqref{H:locstab}.
		\item 	The measurability of the map $\xi\mapsto\int_{\M_\Lambda} F(\gamma) \Xi_\Lambda(\xi,d\gamma)$ follows from the measurability of $\xi\mapsto H(\gamma_\Lambda\xi_{\Lambda^c})$ and $\xi\mapsto Z_\Lambda(\xi)$.
		\item The compatibility of the family $(\Xi_\Lambda)_\Lambda$ follows, as in \cite{preston_1976}, from the additivity \eqref{eq:additivity} of the conditional energy functional.
	\end{enumerate}
\end{proof}
We now state the main result of this subsection:
\begin{Proposition}\label{prop:Gibbs}
	The probability measure $\bar{P}$ is an infinite-volume Gibbs measure with energy functional $H$.
\end{Proposition}
\begin{proof}
	Since $\bar{P}$ is concentrated on the tempered configurations, we have to check that, for any finite-volume $\Lambda$, the following DLR equation is satisfied under $\bar P$:
	\begin{equation*}
  		\int_{\M^{\temp}} F(\gamma)\ \bar{P}(d\gamma) =
\int_{\M^{\temp}}\int_{\M_\Lambda}F(\gamma)\Xi_\Lambda(\xi,d\gamma)\ \bar{P}(d\xi),
	\end{equation*}
	where $F$ is a measurable, bounded and $\Lambda$-local functional.

	Fix $\Lambda\in\mathcal{B}_b(\IR^d)$. We would like to use the fact that its finite-volume approximations $(\bar P_n)_n$ satisfy \ref{DLR}; but since they are lattice-stationary and periodic, this is not true.
	To overcome this difficulty, we use some approximation techniques, articulated in the following three steps:
	\begin{enumerate}[label=\textbf{\roman*.},leftmargin=1.4em]
		\item \textbf{An equivalent sequence:} We introduce a new sequence $(\hat P_n)_n$ and show it is asymptotically equivalent to $(\bar P_n)_n$
		\item \textbf{A cut-off kernel:} We introduce a cut off of the Gibbsian kernel by a local functional
		\item \textbf{Gibbsianity of the limit measure:} We use estimations via the cut-off kernel to prove that $\bar P$ satisfies \ref{DLR}.
	\end{enumerate}
	\medbreak
	
	\begin{enumerate}[label=\textbf{\roman*.},leftmargin=1.4em]
	\item \textbf{An equivalent sequence:} We introduce a modified sequence of measures $(\hat{P}_n)_n$ satisfying \ref{DLR} and having the same asymptotic behaviour as $(\bar{P}_n)_n$: for every $n\geq1$, consider
	\begin{equation*}
  		\hat{P}_n = \dfrac{1}{|\Lambda_n|}\sum_{\substack{\kappa\in \Lambda_n\cap\IZ^d: \\ \theta_\kappa(\Lambda_n)\supset\Lambda}}P_n\circ\theta_\kappa^{-1}.
	\end{equation*}
	Since the above sum is not taken over all $\kappa\in\Lambda_n\cap\IZ^d$, $\hat{P}_n$ is not a probability measure. Moreover, $\hat{P}_n$ is bounded from above by $\bar{P}_n$, in the sense that, for any measurable $A\subset\M$, \mbox{$\hat{P}_n(A)\leq\bar{P}_n(A)$}.
	
	We introduce the index $i_0\in\IN$ as the smallest $n\geq1$ such that $\Lambda$ is contained in the box $\Lambda_{n}$. Using the compatibility of the kernels \eqref{eq:compatibility}, since $\Lambda\subset\Lambda_n$. For every $n\geq i_0$, the measure $\hat{P}_n$ satisfies \ref{DLR}.
	
	The sequences $(\hat{P}_n)_n$ and $(\bar{P}_n)_n$ are locally asymptotically equivalent, in the sense that, for every tame $\Lambda$-local functional $G$ in $\mathcal{L}$,
  	\begin{equation*}
    	\lim\limits_{n\rightarrow\infty}\left\lvert\int G(\gamma)\hat{P}_n(d\gamma)-\int G(\gamma)\bar{P}_n(d\gamma)\right\rvert = 0.
  	\end{equation*}
  	In particular, asymptotically $\hat{P}_n$ is a probability measure, i.e. for any $\epsilon>0$, we can find $n_0$ such that
  	\begin{equation}\label{eq:asymptproba}
   		\forall n\geq n_0,\quad \hat{P}_n(\M) \geq 1-\epsilon.
	\end{equation}
   	Indeed: let $G$ be a tame $\Lambda$-local  functional in $\mathcal{L}$ as in Definition \ref{def:tame}, and set
	\begin{equation*}
	\begin{split}
   		\delta_1 &\defeq \bigg\lvert \int_{\M^{\temp}} G(\gamma)\ \hat{P}_n(d\gamma) - \int_{\M^{\temp}} G(\gamma)\ \bar{P}_n(d\gamma)\bigg\rvert\\
		&=\bigg\lvert \dfrac{1}{(2n)^d} \sum_{\substack{\kappa\in \Lambda_n\cap\IZ^d: \\ \theta_\kappa(\Lambda_n)\supset\Lambda}} \int G(\gamma)P_n\circ\theta_\kappa^{-1}(d\gamma) \\
    &\phantom{= \Big\lvert} - \dfrac{1}{(2n)^d}\sum_{\kappa\in\Lambda_n\cap\IZ^d} \int G(\gamma)\tilde{P}_n\circ\theta_\kappa^{-1}(d\gamma)\Big\rvert
	\end{split}
   	\end{equation*}
   	We then have
   	\begin{equation*}
   	\begin{split}
   		\delta_1 &\leq \dfrac{1}{(2n)^d}\sum_{\substack{\kappa\in\Lambda_n\cap\IZ^d: \\ \theta_\kappa(\Lambda_n)\not\supseteq\Lambda}} \Big\lvert \int G(\gamma)\tilde{P}_n\circ\theta_\kappa^{-1}(d\gamma)\Big\rvert \\
    &\overset{\mathclap{\emph{(loc.+tame)}}}{\leq}\quad \dfrac{c}{(2n)^d}\sum_{\substack{\kappa\in\Lambda_n\cap\IZ^d: \\ \theta_\kappa(\Lambda_n)\not\supseteq\Lambda}} \int \Big(1+\langle\gamma_{\Lambda},1+\norm{m}^{d+\delta}\rangle\Big)\tilde{P}_n\circ\theta_\kappa^{-1}(d\gamma).
	\end{split}
   	\end{equation*}

		As $\Lambda\subset\Lambda_{i_0}$, the number of $\kappa\in\Lambda_n\cap\IZ^d$ such that $\theta_\kappa(\Lambda_n)\not\supseteq\Lambda$ is
  \begin{equation*}
    \text{Card}\big\{\kappa\in\Lambda_n\cap\IZ^d : \theta_\kappa (\Lambda_n)\not\supseteq\Lambda\big\}\leq i_0 2d(2n-1)^{d-1},
  \end{equation*}
  since: for $\Lambda$ to be moved out of $\Lambda_n$, one of the components of $\kappa$ should be larger than $n-i_0$ ($i_0$ options for this); there are $2d$ directions $\Lambda$ can be moved \emph{through} $\Lambda_n$; and the other $d-1$ components of $\kappa\in\Lambda_n\cap\IZ^d$ are left free ($2n-1$ options).

Calling $c'\defeq i_0dc$, we find
  	\begin{equation*}
  	\begin{split}
   		\delta_1 \leq \dfrac{c'}{n} + \dfrac{c}{(2n)^d}\sum_{\substack{\kappa\in\Lambda_n\cap\IZ^d: \\ \theta_\kappa(\Lambda_n)\not\supseteq\Lambda}} \int \langle\gamma_{\Lambda},1+\norm{m}^{d+\delta}\rangle \tilde{P}_n\circ\theta_\kappa^{-1}(d\gamma).
  \end{split}
  \end{equation*}
  Now, for any $a_4>0$ (which will be fixed later), we split the above integral over the set $\big\{\sum_{(x,m)\in\gamma_\Lambda}(1+\norm{m}^{d+\delta})\geq a_4\big\}$ and its complement. We obtain
  \begin{equation*}
  \begin{split}
    \delta_1 &\leq \dfrac{c'}{n} + \dfrac{a_4c'}{n}\\
    &\phantom{\leq \dfrac{c'}{n}\ }+ \dfrac{c}{(2n)^d}\sum_{\substack{\kappa\in\Lambda_n\cap\IZ^d: \\ \theta_\kappa(\Lambda_n) \not\supseteq\Lambda}} \int_{\{\langle\gamma_{\Lambda},1+\norm{m}^{d+\delta}\rangle\geq a_4\}} \langle\gamma_{\Lambda},1+\norm{m}^{d+\delta}\rangle \tilde{P}_n\circ\theta_\kappa^{-1}(d\gamma) \\
    &\leq \underbrace{\vphantom{\int_\{\langle\gamma_{\Lambda},1+\norm{m}^{d+\delta}\rangle \geq a_4\}}\dfrac{(1 + a_4)c'}{n}}_{\text{first term}} + \underbrace{c \int_{\{\langle\gamma_{\Lambda},1+\norm{m}^{d+\delta}\rangle \geq a_4\}} \langle\gamma_{\Lambda},1+\norm{m}^{d+\delta}\rangle \bar{P}_n(d\gamma)}_{\text{second term}}.
  \end{split}
  \end{equation*}
  Fix $\epsilon>0$; for $n\geq\tfrac{2(1+a_4)c'}{\epsilon}$, the first term is smaller than $\epsilon/2$.
  
  To control the second term, we apply Lemma \ref{lemma:EI},
  \iffalse
  recall the following result (due to H.-O. Georgii and H. Zessin, see \cite{georgii_zessin_1993}, Lemma 5.2), which proves that point configurations in $\IR^d$ with marks in a complete, separable metric space, satisfy a local equi-integrability property on entropy level sets, with respect to the marks:
	\emph{For any measurable non-negative function $f:\mathcal{S}\rightarrow\IR_+$ and for every $a>0$ and $\Delta\in\mathcal{B}_b(\IR^d)$,}
  	\begin{equation*}
      	\lim\limits_{N\rightarrow\infty}\sup\limits_{P\in\mathcal{P}(\M)_{\leq a}}\IE_{P}\Big(\1_{\{\langle\gamma_\Delta,f\rangle\geq N\}}\langle\gamma_\Delta,f\rangle\Big) = 0.
    \end{equation*}
    \fi
    to the sequence $(\bar{P}_n)_n$, and the function $f(x,m) = 1+\norm{m}^{d+\delta}$; we find $a_4>0$ such that the second term is smaller than $\epsilon/2$, uniformly in $n$, and conclude the proof of this step.
	\vspace{1em}
	\item \textbf{A cut-off kernel:} We know that $\hat{P}_n$ satisfies \ref{DLR}, i.e. for any $\Lambda$-local and bounded functional $F$
\begin{equation*}
  \int F(\gamma)\ \hat{P}_n(d\gamma) =
\int\int_{\M_\Lambda} F(\gamma)\Xi_\Lambda(\xi,d\gamma)\ \hat{P}_n(d\xi).
\end{equation*}
If $\xi \mapsto \int_{\M_\Lambda}F(\gamma)\Xi_\Lambda(\xi,d\gamma)$ were a local functional, we would be able to conclude simply by taking the limit in $n$ on both sides of the above expression, since \mbox{$\bar{P}= \lim_n \hat{P}_n$} for the topology of local convergence. But this is not the case because of the unboundedness of the range of the interaction.  We are then obliged to consider some approximation tools.
\smallbreak

To that aim, we introduce a $(\Delta,m_0)$-\emph{cut off} of the Gibbsian kernels $\Xi_\Lambda(\xi, d\gamma)$, which takes into account only the points of $\xi$ belonging to a finite volume $\Delta$ and having marks smaller than $m_0$.
\begin{Definition}
  Let $\Delta\in\mathcal{B}_b(\IR^d)$ with $\Delta\supset\Lambda$. The $(\Delta,m_0)$-cut off $ \Xi^{\Delta,m_0}_{\Lambda}$ of the Gibbsian kernel $\Xi_\Lambda$ is defined as follows:\\
	for every measurable, $\Lambda$-local and bounded functional $G\colon\M_\Lambda\rightarrow\IR$,
  \begin{equation*}
  \begin{split}
    \int_{\M_\Lambda}&G(\gamma)\Xi^{\Delta,m_0}_{\Lambda}(\xi,d\gamma)\\
    &\defeq \tfrac{1}{Z_{\Lambda}^{\Delta,m_0}(\xi_{\Delta\setminus\Lambda})}\int_{\M_\Lambda} \1_{\{\tg{m}(\gamma)\leq m_0\}}G(\gamma)e^{-H_\Lambda(\gamma_\Lambda\xi_{\Delta\setminus\Lambda})}\pi^z_\Lambda(d\gamma),
  \end{split}
  \end{equation*}
  where $Z_\Lambda^{^{\Delta,m_0}}(\xi_{\Delta\setminus\Lambda})$ is the normalisation constant.
\end{Definition}

\begin{Remarks}
  \begin{enumerate}[label=\emph{\roman*.}]
    \item $\Xi_\Lambda^{\Delta,m_0}$ is well defined since the normalisation constant $Z_\Lambda^{^{\Delta,m_0}}$ is positive and finite:
    \begin{equation*}
        0 < e^{-z\lvert\Lambda\rvert} \leq Z_{\Lambda}^{\Delta,m_0}(\xi_{\Delta\setminus\Lambda}) < +\infty.
      \end{equation*}
    \item The functional
      \begin{equation*}
        \xi\mapsto\int_{\M_\Lambda} G(\gamma)\, \Xi_\Lambda^{\Delta,m_0}(\xi,d\gamma)
    \end{equation*}
      is now local and bounded, since the supremum norm of $G$ is bounded.
  \end{enumerate}
\end{Remarks}

\vspace{1em}
We now show that $\Xi^{\Delta,m_0}_\Lambda$ is a uniform local approximation of the Gibbsian kernel $\Xi_\Lambda$, i.e.
	
\emph{For any $\epsilon>0, \tg{t}\geq 1$, for any measurable, $\Lambda$-local and bounded functional $F$, there exist $\underline{m}_0>0$ and $\underline{\Delta}\supset\Lambda$ such that, for any $m_0\geq\underline{m}_0$ and $\Delta\supset\underline{\Delta}$, we have}
  \begin{equation} \label{eq:kernestimates}
    \sup\limits_{\xi\in\M^\tg{t}}\bigg\lvert
\int_{\M_\Lambda}F(\gamma)\Xi^{\Delta,m_0}_{\Lambda}(\xi,d\gamma)
-
\int_{\M_\Lambda}F(\gamma)\Xi_\Lambda(\xi,d\gamma)\bigg\rvert\leq\epsilon.
  \end{equation}
  Indeed, let $\xi\in\M^\tg{t}$. First notice that, since $H_\Lambda(\gamma_\Lambda\xi_{\Delta\setminus\Lambda}) = H_\Lambda(\gamma_\Lambda\xi_{\Lambda^c})$ as soon as \mbox{$\Delta \supseteq \Lambda\oplus B\big(0,3\tg{d}(\Lambda)+ 2\tg{l}(\tg{t})+2\tg{m}(\gamma_\Lambda)+1\big)$} then $e^{-H_\Lambda(\gamma_\Lambda\xi_{\Delta\setminus\Lambda})} - e^{-H_\Lambda(\gamma_\Lambda\xi_{\Lambda^c})}=0$ on the set of configurations $\{\gamma: \tg{m}(\gamma_\Lambda) \leq m_0 \textrm{ and }\Lambda\oplus B\big(0,3\tg{d}(\Lambda) + 2\tg{l}(\tg{t})+2\tg{m}(\gamma_\Lambda)+1\big) \subset\Delta\}$.

  Considering the difference between both partition functions, we obtain
  \begin{equation*}
  \begin{split}
    &\left\lvert Z_{\Lambda}^{\Delta,m_0}(\xi_{\Delta\setminus\Lambda}) - Z_\Lambda(\xi_{\Lambda^c})\right\rvert = \left\lvert \int \big(\1_{\{\tg{m}(\gamma_\Lambda) \leq m_0\}} e^{-H_\Lambda(\gamma_\Lambda\xi_{\Delta\setminus\Lambda})} - e^{-H_\Lambda(\gamma_\Lambda\xi_{\Lambda^c})}\big) \pi_\Lambda^z(d\gamma) \right\rvert \\
    &\leq  \int \1_{\{\tg{m}(\gamma_\Lambda) > m_0\}\cup\{\Lambda\oplus B\big(0,3\tg{d}(\Lambda)+ 2\tg{l}(\tg{t})+2\tg{m}(\gamma_\Lambda)+1\big) \not\subseteq\Delta\}} \big( e^{-H_\Lambda(\gamma_\Lambda\xi_{\Delta\setminus\Lambda})} + e^{-H_\Lambda(\gamma_\Lambda\xi_{\Lambda^c})}\big) \pi_\Lambda^z(d\gamma)  \\
    &\overset{\eqref{H:locstab}}{\leq} \int\1_{\{\tg{m}(\gamma_\Lambda) > m_0\}\cup\{\Lambda\oplus B\big(0,3\tg{d}(\Lambda)+ 2 \tg{l}(\tg{t})+2 \tg{m}(\gamma_\Lambda)+1\big) \not\subseteq\Delta\}}2 e^{\tg{c}'\langle\gamma_{\Lambda},1+\norm{m}^{d+\delta}\rangle}\pi^z_\Lambda(d\gamma).
  \end{split}
  \end{equation*}

  Notice that this upper bound does \emph{not} depend on $\xi$ anymore. Thanks to the integrability assumption \eqref{H:mark}, by dominated convergence this implies that the map
  \begin{equation*}
    \xi\mapsto Z_\Lambda^{\Delta,m_0}(\xi_{\Delta\setminus\Lambda}) - Z_\Lambda(\xi_{\Lambda^c})
  \end{equation*}
  converges to $0$ as $m_0\uparrow\infty$ and $\Delta\uparrow\IR^d$
uniformly in $\xi\in\M^{\tg{t}}$.

  Similarly,
  \begin{equation*}
  \begin{split}
    \xi \mapsto \int_{\M_\Lambda} & \1_{\{\tg{m}(\gamma_\Lambda) \leq m_0\}} F(\gamma_\Lambda) e^{-H_\Lambda(\gamma_\Lambda\xi_{\Delta\setminus\Lambda})}\pi^z_\Lambda(d\gamma) \\
    &- \int_{\M_\Lambda} F(\gamma_\Lambda)e^{-H_\Lambda(\gamma_\Lambda\xi_{\Lambda^c})}\pi^z_\Lambda(d\gamma)
  \end{split}
  \end{equation*}
  converges to $0$ as $m_0\uparrow\infty$ and $\Delta\uparrow\IR^d$, uniformly in $\xi\in\M^{\tg{t}}$. This concludes the proof of this step: we can find $\underline{m}_0=\underline{m}_0(\epsilon,\tg{t})$ and $\underline{\Delta}=\underline{\Delta}(\epsilon,\tg{t})$ such that \eqref{eq:kernestimates} holds for any $m_0\geq \underline{m}_0$ and $\Delta\supset\underline{\Delta}$.

\item \textbf{Gibbsianity of the limit measure:} To prove the Gibbsianity of $\bar P$ we have to check that
  	\begin{equation*}
    	\delta_2 \defeq \left\lvert\int_{\M^{\temp}}\int_{\M_\Lambda} F(\gamma)\Xi_\Lambda(\xi,d\gamma)\ \bar{P}(d\xi) - \int_{\M^{\temp}} F(\gamma)\ \bar{P}(d\gamma)\right\rvert
  	\end{equation*}
  	vanishes.
  	
	We first show that for large enough $\tg{t}\geq 1$, the sets $\M^\tg{t},\underline{\M}^{\tg{l}(\tg{t})}$ are close to the support of the measures $\bar{P}$, $\bar P_n$, and $\hat P_n$: let $m_0$ and $\Delta$ be large enough in the above sense, and satisfy \mbox{$\Delta\supset\Lambda\oplus B\big(0,3\tg{d}(\Lambda)+ 2\tg{l}(\tg{t})+2m_0+1\big)$}. Thanks to the results on the supports of $\bar{P}$ (Proposition \ref{prop:supbarP}) and $\bar{P}_n$ (Proposition \ref{prop:supbarPn}), we can find $a_5>0$, independent of $n$, such that, for any $m_0$ and $\tg{t}$ larger than $a_5$, and all $n\geq1$,
	\begin{equation}\label{eq:support}
    	\bar{P}(\M^\tg{t}) \geq 1-\epsilon,\ \bar{P}_n(\underline{\M}^{\tg{l}(\tg{t})})\geq 1-\epsilon,\ \bar{P}_n\Big(\big\{\gamma\in\M:\ \tg{m}(\gamma_\Lambda)\leq m_0\big\}\Big) \overset{\eqref{eq:boundm}}{\geq} 1-\epsilon.
  	\end{equation}
 	Since, by construction, $\bar{P}_n$ dominates $\hat{P}_n$, using \eqref{eq:asymptproba} yields, for $n\geq n_0$,
   	\begin{equation}\label{eq:estimatesb}
		\hat{P}_n(\underline{\M}^{\tg{l}(\tg{t})})\geq 1-2\epsilon, \quad \hat{P}_n\Big(\big\{\gamma\in\M:\ \tg{m}(\gamma_\Lambda)\leq m_0\big\}\Big) \geq 1-2\epsilon.
	\end{equation}
	The following steps deal with the estimation of $\delta_2$: using \eqref{eq:support}, we have that (w.l.o.g. $\lVert F\rVert_\infty\leq 1$)
  	\begin{equation*}
    	\delta_2 \leq \underbrace{\lVert F\rVert_\infty \bar{P}\big((\M^{\tg{t}})^c\big) \vphantom{\int_{\M_\Lambda}}}_{\leq\epsilon} + \underbrace{\left\lvert \int_{\M^{\tg{t}}}\int_{\M_\Lambda}F(\gamma)\Xi_\Lambda(\xi,d\gamma)\bar{P}(d\xi) - \int_{\M^{\tg{t}}}F(\gamma)\bar{P}(d\gamma) \right\rvert}_{=:\delta_{21}}.
  	\end{equation*}
  	
	Using the estimates of \eqref{eq:kernestimates}, we have
	\begin{equation*}
   	\begin{split}
    	\delta_{21} & \overset{\mathclap{\eqref{eq:kernestimates}}}{\leq} \epsilon + \bigg\lvert \int_{\M^\tg{t}}\int_{\M_\Lambda}F(\gamma)\Xi_\Lambda^{\Delta,m_0}(\xi,d\gamma)\bar{P}(d\xi) - \int_{\M^\tg{t}}F(\gamma)\bar{P}(d\gamma)\bigg\rvert \\
		&\overset{\mathclap{\eqref{eq:support}}}{\leq} 2\epsilon + \underbrace{\bigg\lvert \int_{\M^{\temp}}\int_{\M_\Lambda}F(\gamma)\Xi_\Lambda^{\Delta,m_0}(\xi,d\gamma)\bar{P}(d\xi) - \int_{\M^{\temp}}F(\gamma)\bar{P}(d\gamma)\bigg\rvert}_{=:\delta_{22}}.
   	\end{split}
   	\end{equation*}
	By construction, the functional $\xi\mapsto\int_{\M_{\Lambda}}F(\gamma)\Xi_\Lambda^{\Delta,m_0}(\xi,d\gamma)$ is local; thus the local convergence of $(\hat{P_n})_n$ to $\bar P$ implies that there exists $n_1\in\IN^*$ such that, for $n\geq n_1$, both estimates hold:
	\begin{gather*}
		\left\lvert \int_{\M^{\temp}}\int_{\M_\Lambda}F(\gamma)\Xi_\Lambda^{\Delta,m_0}(\xi,d\gamma)\bar{P}(d\xi) - \int_{\M^{\temp}}\int_{\M_\Lambda}F(\gamma)\Xi_\Lambda^{\Delta,m_0}(\xi,d\gamma)\hat{P}_n(d\xi) \right\rvert \leq \epsilon,\\
		\left\lvert \int_{\M^{\temp}}F(\gamma)\bar{P}(d\gamma) - \int_{\M^{\temp}}F(\gamma)\hat{P}_n(d\gamma) \right\rvert \leq \epsilon.
	\end{gather*}
    Therefore,
    \begin{equation*}
     			\delta_{22} \leq 2\epsilon + \underbrace{\bigg\lvert \int_{\M^{\temp}}\int_{\M_\Lambda}F(\gamma)\Xi_\Lambda^{\Delta,m_0}(\xi,d\gamma)\hat{P}_n(d\xi) - \int_{\M^{\temp}}F(\gamma)\hat{P}_n(d\gamma)\bigg\rvert}_{=:\delta_{23}}.
	\end{equation*}
	Now $\delta_{23}$ can be further decomposed:
	\begin{equation*}
	\begin{split}
		\delta_{23}\overset{\eqref{eq:estimatesb}}{\leq} 2\epsilon &+ \underbrace{\bigg\lvert \int_{\underline\M^{{\tg{l}(\tg{t})}}}\int_{\M_\Lambda}F(\gamma)\Xi_\Lambda^{\Delta,m_0}(\xi,d\gamma)\hat{P}_n(d\gamma) - \int_{\underline\M^{\tg{l}(\tg{t})}}\int_{\M_\Lambda}F(\gamma)\Xi_\Lambda(\xi,d\gamma)\hat{P}_n(d\xi)\bigg\rvert}_{=:\delta_{24}} \\
		&+ \left\lvert\int_{\underline\M^{\tg{l}(\tg{t})}} \int_{\M_\Lambda}F(\gamma)\Xi_\Lambda(\xi,d\gamma)\ \hat{P}_n(d\xi)-\int_{\M^{\temp}} F(\gamma) \hat{P}_n(d\gamma)\right\rvert.
	\end{split}
	\end{equation*}
	We can estimate $\delta_4$ by conditioning
	\begin{equation*}
	\begin{split}
		\delta_{24}&= \bigg\lvert \int_{\underline\M^{{\tg{l}(\tg{t})}}}\int_{\M_\Lambda}F(\gamma)\Big(\Xi_\Lambda^{\Delta,m_0}(\xi,d\gamma) - \Xi_\Lambda(\xi,d\gamma)\Big)\hat{P}_n(d\gamma)\bigg\rvert\\
		&= \bigg\rvert \int_{\underline\M^{\tg{l}(\tg{t})}}\int_{\M_\Lambda}F(\gamma)\Big(\Xi^{\Delta,m_0}_{\Lambda}(\xi,d\gamma)\hat{P}_n(d\xi)\\
		&\phantom{\bigg\lvert\int_{\M^{\tg{t}}}} -\Xi_\Lambda\big(\xi,d\gamma\ \vert\ \{\gamma:\tg{m}(\gamma)\leq m_0\}\big)\big(1-\Xi_\Lambda\big(\xi,\{\gamma':\tg{m}(\gamma')> m_0\}\big)\big) \\
		&\phantom{\bigg\lvert\int_{\M^{\tg{t}}}\int} + \Xi_\Lambda\big(\xi,d\gamma\ \vert\ \{\gamma:\tg{m}(\gamma) > m_0\}\big)\ \Xi_\Lambda\big(\xi,\{\gamma':\tg{m}(\gamma')> m_0\}\big) \Big)\hat{P}_n(d\xi) \\
		&\leq \bigg\lvert\int_{\underline\M^{\tg{l}(\tg{t})}}\int_{\M_\Lambda}F(\gamma)\Big(\Xi^{\Delta,m_0}_{\Lambda}(\xi,d\gamma) - \Xi_\Lambda\big(\xi,d\gamma\ \vert\ \{\gamma : \tg{m}(\gamma)\leq m_0\}\big)\Big)\hat{P}_n(d\xi)\bigg\rvert\\
		&\phantom{\bigg\lvert\int_{\M}} + 2 \int_{\underline\M^{\tg{l}(\tg{t})}}\Xi_\Lambda\big(\xi,\{\gamma':\tg{m}(\gamma')> m_0\}\big)\hat{P}_n(d\xi).
	\end{split}
	\end{equation*}

	The first term in the above inequality vanishes if the two kernels coincide on $\underline\M^{\tg{l}(\tg{t})}$, which is the case thanks to $(\mathcal{H}_{r})$, since $\Lambda\oplus B\big(0,3\tg{d}(\Lambda)+ 2\tg{l}(\tg{t})+2\tg{m}(\gamma_\Lambda)+1\big) \subset\Delta$. Since $\hat{P}_n$ satisfies \ref{DLR}, for the second term we have
   	\begin{equation*}
		\int_{\underline\M^{\tg{l}(\tg{t})}}\Xi_\Lambda\big(\xi,\{\gamma':\tg{m}(\gamma')> m_0\}\big)\hat{P}_n(d\xi) \leq \hat{P}_n\big(\{\gamma':\tg{m}(\gamma'_\Lambda)> m_0\}\big)\overset{\eqref{eq:estimatesb}}{\leq}2\epsilon.
	\end{equation*}
	We then have \mbox{$\delta_{24} \leq 4\epsilon$}.
	Putting it all together,
	\begin{equation*}
	\begin{split}
		\delta_2 &\leq 11\, \epsilon + \left\lvert\int_{\underline\M^{\tg{l}(\tg{t})}} \int_{\M_\Lambda}F(\gamma)\Xi_\Lambda(\xi,d\gamma)\ \hat{P}_n(d\xi)-\int_{\M^{\temp}} F(\gamma)\hat{P}_n(d\gamma)\right\rvert \\
		&\overset{\mathclap{\eqref{eq:estimatesb}}}{\leq} 13\, \epsilon + \left\lvert\int_{\M^{\temp}} \int_{\M_\Lambda}F(\gamma)\Xi_\Lambda(\xi,d\gamma)\ \hat{P}_n(d\xi)-\int_{\M^{\temp}} F(\gamma)\hat{P}_n(d\gamma)\right\rvert = 13\, \epsilon,
	\end{split}
	\end{equation*}
	since $\hat{P}_n$ satisfies \ref{DLR}. Thanks to the arbitrariness of $\epsilon>0$, we can conclude that also $\bar{P}$ satisfies \ref{DLR}.
\end{enumerate}
\vspace{.5em}
In conclusion, $\bar{P}$ satisfies \ref{DLR} for any finite volume $\Lambda$, so Proposition \ref{prop:Gibbs} -- and consequently Theorem \ref{thm:Gibbs} -- is proved: $\bar{P}$ is an infinite-volume Gibbs measure with energy functional $H$.
\end{proof}

\section{Application to infinite-dimensional interacting diffusions}\label{section:diffusion}

We consider the case where the space of marks $\mathcal{S}$ is $C_0\big([0,1],\IR^2\big)$, the set of $\IR^2$-valued continuous paths on $[0,1]$ starting at $0$, endowed with the supremum norm  $\lVert m(\cdot)\rVert:=\max_{s\in[0,1]}\lvert m(s)\rvert $.\\
In other words, a marked point \mbox{$\msb{x}=\left(x,m(\cdot)\right)\in\IR^2\times C_0\big([0,1],\IR^2\big)$} is identified with the continuous path $(x + m(s), s \in [0,1])$ starting in $x$.\\
The random evolution of a reference path follows a gradient dynamics which solves the following Langevin stochastic differential equation:
\begin{equation}\label{eq:SDE}
  dX_s = -\dfrac{1}{2}\grad V(X_s)ds + dB_s,\ s\in [0,1],
\end{equation}
where $B$ is an $\IR^2$-valued Brownian motion.

Since we are looking for a random mark whose norm admits a reference law with a super-exponential moment (Assumption $(\mathcal{H}_m)$ in Section \ref{section:marks}), we shall restrict our attention to potentials $V$ which force the gradient dynamics to be strongly confined. Let us thus assume that the potential $V$ is smooth (i.e. of class $\mathcal{C}^2$) and that it satisfies the following bounds outside some compact set of $\IR^2$:
\begin{equation}\label{eq:VgradV}
  \exists \delta',\tg{a}_1,\tg{a}_2>0, \quad  V(x) \geq \tg{a}_1 \lvert x\rvert^{2+\delta'} \textrm{ and } \Delta V(x) -\frac{1}{2}\lvert \grad V (x)\rvert^2 \leq -\tg{a}_2\vert x\rvert^{2+2\delta'}.
\end{equation}
Indeed, the bounds in \eqref{eq:VgradV} imply the conditions of \cite{royer_2007} for existence and uniqueness of a strong solution of the Langevin equation \eqref{eq:SDE}; moreover, they also ensure the ultracontractivity of this diffusion with respect to its invariant measure \mbox{$\mu(dy) = e^{-V(y)}dy$} (see Example 3.5 in \cite{kavian_1993}).
In particular, this implies (see Theorem 4.7.1 in \cite{davies_1989}) that there exists a constant $\tg{a}>0$ such that, for any $\delta < \delta'/2$,
\begin{equation*}
  \sup_{s\in[0,1]} \IE\left[ e^{\lvert X_s\rvert^{2+2\delta}}\vert X_0 = 0 \right]
\leq \tg{a} \int_{\IR^2} e^{\lvert y\rvert^{2+2\delta}}\mu(dy)
= \tg{a} \int_{\IR^2} e^{\lvert y\rvert^{2+2\delta}}e^{-V(y)}dy <+\infty.
\end{equation*}
We denote by $\rho$ the law on $\IR_+$ of the supremum norm of the Langevin diffusion starting in 0; since the process $Y_t \defeq \sup_{s\in[0,t]}\lvert X_s\rvert^{2+\delta}$ is a submartingale, we can apply Doob's inequality to get
\begin{equation*}
   \int_{\IR_+} e^{\marknorm^{2+2\delta}} \rho(d\marknorm) = \IE\left[ e^{\sup_{s\in[0,1]}\lvert X_s\rvert^{2+2\delta}} \vert X_0 = 0 \right] < +\infty.
\end{equation*}
\begin{Remark}
  The previous reasoning can be generalised by considering the evolution of the Langevin dynamics in $\IR^d$ for any $ d>2$. Assumption \eqref{eq:VgradV} should then be reinforced by replacing the $(2+\delta')$-exponent with a $(d+\delta')$ exponent, in order to obtain the finiteness of the super-exponential moment $(\mathcal{H}_m)$.
\end{Remark}

Let us now describe the kind of interaction we consider between the marked points of a configuration. It is not necessarily superstable, and consists of a pair interaction, concerning separately the (starting) points and their attached diffusion paths, and a self interaction.

The energy of a finite configuration $\gamma = \{\msb{x}_1,\dots,\msb{x}_N\}$ is taken of the form
\begin{equation}\label{eq:diffEn}
  H(\gamma) = \sum_{i=1}^N \Psi(\msb{x}_i) + \sum_{i=1}^N\sum_{j<i}\Phi(\msb{x}_i,\msb{x}_j),
\end{equation}
where
\begin{itemize}
	\item the \emph{self-potential} term $\Psi$ satisfies $\inf_{x\in\IR^2}\Psi(x+m)\geq-\tg{c}_1(1+\norm{m}^{2+\delta})$, for some constant $\tg{c}_1>0$;
	\item the \emph{two-body potential} $\Phi$ is defined by \footnote{An earlier version of this work used a potential $\Phi$ which did not satisfy assumption $(\mathcal{H}_{loc.st})$.}
	\begin{equation}\label{eq:diff:2body}
			\Phi(\msb{x}_i,\msb{x}_j) = \Big(\int_0^{1} \phi(\abs{x_i -x_j + m_i(s) - m_j(s)})ds\Big) \1_{[0, a_0 + \norm{m_i} + \norm{m_j}]}(\abs{x_i-x_j}),
	\end{equation}
	with $\phi$ given by the sum of two potentials on $\IR_+$: $\phi = \phi_{hc} + \phi_{l}$, where
	\begin{itemize}
			\item The potential $\phi_{hc}$ is pure \emph{hard core} at some diameter $R>0$, that is
			\begin{equation*}
				\phi_{hc}(u)=(+\infty)\1_{[0,R)}(u).	
			\end{equation*}
			\item The potential $\phi_l$ satisfies a stability property, i.e. there exists a constant $\tg{c}_\phi\geq 0$ such that, for any admissible configuration $\{y_1,\dots,y_N\}$, $N\geq 1$, the following holds (see \cite{ruelle_1969}, paragraph 3.2.5):
		\begin{equation}\label{eq:diff:hcstab}
			\sum_{i=1}^N \phi_l(\abs{y_i})\geq -2\tg{c}_\phi,
		\end{equation}
		where a finite configuration $\{y_1,\dots,y_N\}\subset\IR^d$, $N\geq 1$, is called \emph{admissible} if, for any pair $y_i\neq y_j$, $\phi(\abs{y_i-y_j})<+\infty$.
		
		Moreover, we assume that $\phi_l(\lvert x\rvert)\leq 0$ for any $\lvert x\rvert\geq a_0$ (see, e.g., Figure \ref{fig:potential}).
		\end{itemize}
		Note that \eqref{eq:diff:hcstab} implies that the potential $\phi$ -- defined on the location space $\IR^d$ -- is \emph{stable} in the sense of Ruelle (see \cite{ruelle_1970} and \cite{minlos_1967}), with stability constant $\tg{c}_\phi$, i.e.
		\begin{equation*}
			\forall N\geq 1,\, \forall\{y_1,\dots,y_N\}\subset\IR^d,\quad \sum_{1\leq i<j\leq N} \phi(\abs{y_i-y_j})\geq - \tg{c}_\phi N.
		\end{equation*}

%	\begin{equation*}
%		\Phi(\msb{x}_1,\msb{x}_2) = \Big(\phi(\lvert x_1-x_2\rvert) + \int_0^1 \tilde{\phi}(\lvert m_1(s)-m_2(s)\rvert)ds\Big) \1_{\{\lvert x_1-x_2\rvert\leq a_0 + \norm{m_1} + \norm{m_2}\}},
%	\end{equation*}
%	where $\phi:\IR_+\rightarrow\IR\cup\{+\infty\}$ is a \emph{stable} pair potential, i.e. there exists a constant $\tg{c}_\phi>0$ such that, for any $\{x_1,\dots,x_N\}\subset\IR^2$, $N\geq 1$,
%	\begin{equation*}
%		\sum_{i=1}^N\sum_{j<i}\phi(\lvert x_i-x_j\rvert)\geq -\tg{c}_\phi N,
%	\end{equation*}
%see \cite{ruelle_1970} and \cite{minlos_1967}; moreover, we assume that $\phi(\lvert x\rvert)\leq 0$ for any $\lvert x\rvert\geq a_0$ (see, e.g., Figure \ref{fig:potential}). Moreover, we assume that $\phi$ is bounded from below by some constant $-\tg{c}_2$. The pair potential $\tilde\phi:\IR^2\rightarrow\IR_+\cup\{+\infty\}$ is a non-negative function.
\end{itemize}

\begin{figure}[t]
  \begin{minipage}[t]{0.45\textwidth}
  	\centering
    \includegraphics[height=\linewidth]{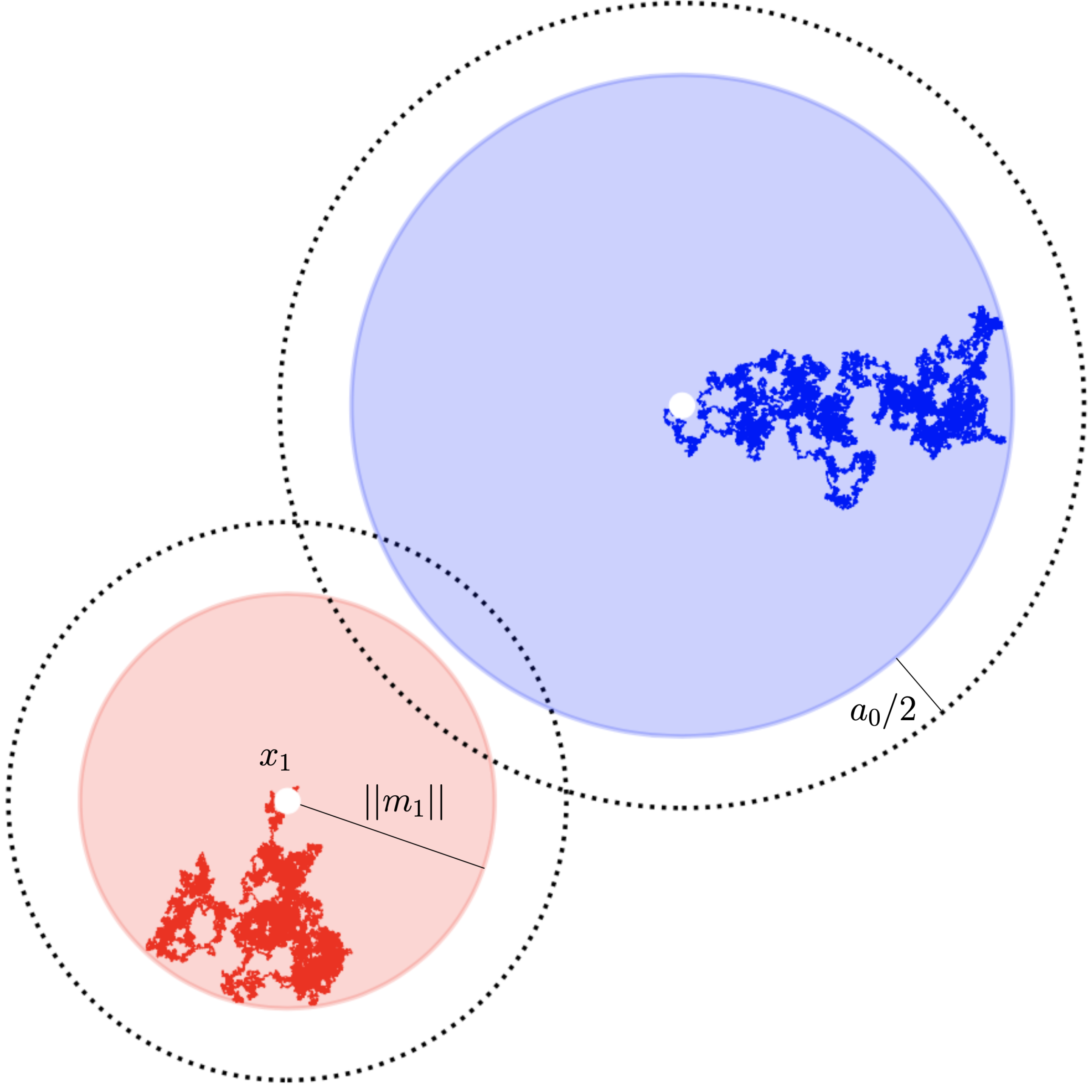}
    \caption{Two paths of a Langevin diffusion in $\IR^2$. Each circle is centred in the starting point; the radius of the coloured circles correspond to their maximum displacement in the time interval $[0,1]$; the dotted circles represent the security distance $a_0/2$.}
    \label{fig:diffs}
  \end{minipage}\hfill
  \begin{minipage}[t]{0.45\textwidth}
  	\centering
  	\includegraphics[height=.6\linewidth]{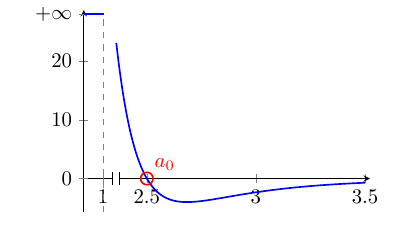}
	\caption{A shifted Lennard--Jones potential $\phi_{LJ}(u-1) = 16\big(\big(\frac{3/2}{u-1}\big)^{12} - \big(\frac{3/2}{u-1}\big)^{6}\big)$ with hard-core diameter $R = 1$; it is always negative after $a_0=2.5$, and explodes as $x\to 1^+$.}
	\label{fig:potential}
  \end{minipage}
\end{figure}

Under these assumptions, the potential $\Phi$ is \emph{stable} $\tg{c}_\Phi = \tg{c}_\phi$ in the following sense:
\begin{equation*}
	\sum_{\{\msb{x},\msb{y}\}\subset\gamma}\Phi(\msb{x},\msb{y}) \geq -\tg{c}_\Phi\lvert\gamma\rvert, \quad \forall \gamma\in\mathcal{M}_f.
\end{equation*}
It is then straightforward to prove that such an energy functional $H$ satisfies the stability assumption $(\mathcal{H}_{st})$ with constant $\tg{c} = \tg{c}_1\vee\tg{c}_\Phi$, and the range assumption $(\mathcal{H}_{r})$. Moreover, the local uniform-stability assumption $(\mathcal{H}_{loc.st})$ also holds:

Let $\gamma\in\mathcal{M}$ and $\xi\in\mathcal{M}^{\tg{t}},\ \tg{t}\geq 1$, and denote $\Delta = \Lambda\oplus B(0,\tg{r})$. We have
\begin{equation*}
  H_{\Lambda}(\gamma_\Lambda\xi_{\Delta\setminus\Lambda}) =  \underbrace{\sum_{\msb{x}\in\gamma_\Lambda} \Psi(\msb{x}) +  \sum_{\{\msb{x},\msb{y}\}\subset\gamma_\Lambda}\Phi(\msb{x},\msb{y})}_{\overset{(\mathcal{H}_{st})}{\geq} -(\tg{c}_1\vee\tg{c}_\Phi)\sum_{\msb{x}\in\gamma_\Lambda} ( 1+\lVert m\rVert^{2+\delta})}  + \sum_{\substack{\msb{x}\in\gamma_\Lambda\\ \msb{y}\in\xi_{\Delta\setminus\Lambda}}}\Phi(\msb{x},\msb{y}).
\end{equation*}
%Since $\xi\in\mathcal{M}^{\tg{t}}$, there exists $\tg{c}_4(\tg{t})$ such that \mbox{$\lvert\xi_\Delta\rvert \leq \tg{c}_4(\tg{t})$}. Therefore,
%\begin{equation*}
%\begin{split}
%  \sum_{\substack{\msb{x}\in\gamma_\Lambda\\ \msb{y}\in\xi_{\Delta\setminus\Lambda}}}\Phi(\msb{x},\msb{y}) &\geq \sum_{\substack{\msb{x}\in\gamma_\Lambda\\ \msb{y}\in\xi_{\Delta\setminus\Lambda}}}\phi(\lvert x-y\rvert) + \sum_{\substack{\msb{x}\in\gamma_\Lambda\\ \msb{y}\in\xi_{\Delta\setminus\Lambda}}} \int_0^1 \tilde{\phi}\big(\lvert m_x(s)-m_y(s)\rvert\big) ds\\
%  & \geq \sum_{\substack{\msb{x}\in\gamma_\Lambda\\ \msb{y}\in\xi_{\Delta\setminus\Lambda}}}\phi(\lvert x-y\rvert) \geq - \tg{c}_2\ \lvert\gamma_\Lambda\rvert\lvert\xi_{\Delta\setminus\Lambda}\rvert\\
%  &\overset{\mathclap{(\xi\in\mathcal{M}^{\tg{t}})}}{\geq} \quad\!  - \bar{\tg{c}}(\tg{t}) \lvert \gamma_\Lambda\rvert,
%\end{split}
%\end{equation*}
%where $\bar{\tg{c}}(\tg{t})\defeq \tg{c}_2\tg{c}_4(\tg{t})$, so that $(\mathcal{H}_{loc.st})$ holds with $\tg{c}'(\tg{t})\defeq \tg{c}_1 \vee \bar{\tg{c}}(\tg{t})$.
We can assume that $\xi_\Delta$ is of finite energy, and therefore use \eqref{eq:diff:hcstab} to estimate
	\begin{equation*}
	\begin{split}
		\sum_{\msb{y}\in\xi_{\Delta\setminus\Lambda}} \Phi(\msb{x},\msb{y}) &= \int_0^{1} \sum_{(y,n)\in\xi_{\Delta\setminus\Lambda}} \phi(\abs{x-y+m(s)-n(s)})ds\, \1_{\{\lvert x - y\rvert\leq a_0 + \norm{m} + \norm{n}\}}\\
		&\overset{\mathclap{\eqref{eq:diff:hcstab}}}{\geq}\  \int_0^{1}  -2\tg{c}_\phi \geq -2\tg{c}_\phi. 
	\end{split}
	\end{equation*}
	Together with the stability of $\gamma_\Lambda\mapsto H(\gamma_\Lambda)$, this yields the following lower bound for the conditional energy:
	\begin{equation*}
		H_\Lambda(\gamma_\Lambda\xi_{\Lambda^c})\geq -(\tg{c}_1\vee 2\tg{c}_\Phi)\sum_{\msb{x}\in\gamma_\Lambda}(1+\norm{m}^{d+\delta}).\qedhere
	\end{equation*}
\begin{Remark}
  The above is an example of a pair potential with finite but not uniformly bounded range.
\end{Remark}
\begin{Example}
  A concrete example of functions satisfying \eqref{eq:VgradV} -- \eqref{eq:diffEn} is as follows:
  \begin{itemize}
    \item 	For the Langevin dynamics, consider $V(x)=\lvert x\rvert^4$; then the diffusion is ultracontractive with $\delta'=2$. The invariant measure $\mu(dx)=e^{\lvert x\rvert^4}dx$ is a Subbotin measure (see \cite{subbotin_1923}).
    \item For the interaction, let \mbox{$\psi(z)=-1-z^{\sfrac{5}{2}}$}, and $\phi$ be given by the sum of a hard-core component and a shifted Lennard--Jones potential, i.e.
		\begin{equation*}
			\phi(u) = \phi_{hc}(u) + \phi_{LJ}(u-R)\1_{[R,+\infty)}(u),\ u\in\IR_+,
		\end{equation*}
		where $\phi_{LJ}(u) = \frac{a}{u^{12}} - \frac{b}{u^6}$, $a,b>0$.
  \end{itemize}
\end{Example}

\begin{acknowledgments}
The authors would like to warmly thank David Dereudre for the fruitful discussions they had together on the topic. We also thank the two referees for the accurate reading and the many comments and remarks, which notably improved the earlier versions of this paper.
\end{acknowledgments}

%\newpage
\bibliographystyle{abbrv}

\begin{thebibliography}{10}

\bibitem{conache_daletskii_kondratiev_pasurek_2018}
D.~Conache, A.~Daletskii, Y.~Kondratiev, and T.~Pasurek.
\newblock Gibbs states of continuum particle systems with unbounded spins:
  Existence and uniqueness.
\newblock {\em J. Math. Phys.}, 59(1):013507, 2018.

\bibitem{cramer_1938}
H.~Cram\'{e}r.
\newblock {S}ur un nouveau th\'{e}or\`{e}me-limite de la th\'{e}orie des
  probabilit\'{e}s.
\newblock {\em Colloque consacr\'{e} \`{a} la th\'{e}orie des probabilit\'{e}s.
  Actual. sci. ind.}, 736:5--23, 1938.

\bibitem{daletskii_kondratiev_kozitsky_pasurek_2014}
A.~Daletskii, Y.~Kondratiev, Y.~Kozitsky, and T.~Pasurek.
\newblock Gibbs states on random configurations.
\newblock {\em J. Math. Phys.}, 55(8):083513, 2014.

\bibitem{davies_1989}
E.~W. Davies.
\newblock {\em Heat kernels and spectral theory}, volume~92 of {\em Camb.
  tracts Math.}
\newblock Cambridge University Press, 1989.

\bibitem{dereudre_2009}
D.~Dereudre.
\newblock The existence of {Q}uermass-interaction processes for nonlocally
  stable interaction and nonbounded convex grains.
\newblock {\em Adv. Appl. Probab.}, 41(03):664--681, 2009.

\bibitem{dereudre_2019}
D.~Dereudre.
\newblock {\em Introduction to the Theory of Gibbs Point Processes}, pages
  181--229.
\newblock Springer International Publishing, Cham, 2019.

\bibitem{dereudre_drouilhet_georgii_2011}
D.~Dereudre, R.~Drouilhet, and H.-O. Georgii.
\newblock Existence of {G}ibbsian point processes with geometry-dependent
  interactions.
\newblock {\em Probab. Th. Rel. Fields}, 153(3-4):643--670, 2011.

\bibitem{dereudre_vasseur_2019}
D.~Dereudre and T.~Vasseur.
\newblock Existence of {G}ibbs point processes with stable infinite range
  interaction.
\newblock {\em To appear in Adv. Appl. Probab.}, 2020.

\bibitem{foellmer_1973}
H.~F\"{o}llmer.
\newblock On entropy and information gain in random fields.
\newblock {\em Z. Wahrscheinlichkeitstheorie und verw. Gebiete},
  26(3):207--217, 1973.

\bibitem{georgii_1979}
H.-O. Georgii.
\newblock {\em Canonical {G}ibbs Measures}, volume 760 of {\em Lect. Notes
  Math.}
\newblock Springer, 1979.

\bibitem{georgii_2011}
H.-O. Georgii.
\newblock {\em Gibbs measures and phase transitions}.
\newblock De Gruyter studies in Mathematics 9. De Gruyter, 2nd edition, 2011.

\bibitem{georgii_zessin_1993}
H.-O. Georgii and H.~Zessin.
\newblock Large deviations and the maximum entropy principle for marked point
  random fields.
\newblock {\em Probab. Th. Rel. Fields}, 96(2):177--204, 1993.

\bibitem{kahane_1960}
J.~Kahane.
\newblock Propri\'{e}t\'{e}s locales des fonctions \`{a} s\'{e}ries de
  {F}ourier al\'{e}atoires.
\newblock {\em Studia Math.}, 19(1):1--25, 1960.

\bibitem{kavian_1993}
O.~Kavian, G.~Kerkyacharian, and B.~Roynette.
\newblock Quelques remarques sur l'ultracontractivit{\'e}.
\newblock {\em J. Funct. Anal.}, 111:155--196, 1993.

\bibitem{kendall_van_lieshout_baddeley_1999}
W.~S. Kendall, M.~N.~M. van Lieshout, and A.~J. Baddeley.
\newblock Quermass-interaction processes: conditions for stability.
\newblock {\em Adv. Appl. Probab.}, 31(2):315--342, 1999.

\bibitem{ledoux_talagrand_2012}
M.~Ledoux and M.~Talagrand.
\newblock {\em Probability in {B}anach spaces: isoperimetry and processes},
  volume~23 of {\em Ergeb. Math. Grenzgeb.}
\newblock Springer, 1991.

\bibitem{minlos_1967}
R.~A. Minlos.
\newblock Regularity of {G}ibbs limit distribution.
\newblock {\em Funct. Anal. Appl.}, 1(3):206--217, Jul 1967.

\bibitem{nguyen_zessin_1979}
X.~X. Nguyen and H.~Zessin.
\newblock Ergodic theorems for spatial processes.
\newblock {\em Z. Wahrscheinlichkeitstheorie und verw. Gebiete},
  48(2):133--158, 1979.

\bibitem{preston_1976}
C.~Preston.
\newblock {\em Random Fields}, volume 534 of {\em Lect. Notes Math.}
\newblock Springer, 1976.

\bibitem{royer_2007}
G.~Royer.
\newblock {\em An initiation to logarithmic Sobolev inequalities}, volume~14 of
  {\em SMF/AMS Texts Monogr.}
\newblock Am. Math. Soc., 2007.

\bibitem{ruelle_1969}
D.~Ruelle.
\newblock {\em Statistical mechanics: rigorous results}.
\newblock W. A. Benjamin, Inc., New York-Amsterdam, 1969.

\bibitem{ruelle_1970}
D.~Ruelle.
\newblock Superstable interactions in classical statistical mechanics.
\newblock {\em Comm. Math. Phys.}, 18(2):127--159, 1970.

\bibitem{subbotin_1923}
M.~T. Subbotin.
\newblock On the law of frequency of error.
\newblock {\em Mat. Sb.}, 31:296--301, 1923.

\end{thebibliography}

\end{document}